\documentclass[11pt,a4paper,reqno]{amsart}

\usepackage{fourier}

\makeatletter
\def\resetMathstrut@{%
  \setbox\z@\hbox{%
    \mathchardef\@tempa\mathcode`\(\relax
    \def\@tempb##1"##2##3{\the\textfont"##3\char"}%
    \expandafter\@tempb\meaning\@tempa \relax
  }%
  \ht\Mathstrutbox@1.2\ht\z@ \dp\Mathstrutbox@1.2\dp\z@
}
\makeatother

\usepackage[lf]{venturis}
\usepackage{stmaryrd}   

\usepackage{tikz}

\usepackage{fullpage}
\usepackage[latin1]{inputenc}
\usepackage{hyperref}   
\usepackage{enumitem}

\swapnumbers
\theoremstyle{definition}
\newtheorem{remark}{Remark}[section]

\newtheorem{definition}[remark]{Definition}

\newtheorem{notation}[remark]{Notation}

\newtheorem{example}[remark]{Example}

\newtheorem{conventions}[remark]{Conventions}

\theoremstyle{plain}
\newtheorem{theorem}[remark]{Theorem}
\newtheorem{lemma}[remark]{Lemma}
\newtheorem{factit}[remark]{Fact}

\newtheorem{corollary}[remark]{Corollary}

\newtheorem{Klingenberg}[remark]{Klingenberg's lemma}

\newcommand{\comment}[1]{}            
\newcommand{\scare}[1]{``#1''}        
\newcommand{\mt}[1]{{\text{\rm #1}}}  
\newcommand{\dif}{\mt{d}}             
\newcommand{\diff}{\mathop{}\!\dif}   
\newcommand{\e}{\mathord{\mt{e}}}     

\newcommand{\eps}{\varepsilon}        

\newcommand{\set}[1]{\left\{#1\right\}}

\newcommand{\bigset}[1]{\big\{#1\big\}}

\newcommand{\Set}[2]{\left\{#1\mathrel{}\middle|\mathrel{}#2\right\}}

\newcommand{\bigSet}[2]{\big\{#1\mathbin{\big|}#2\big\}}
\newcommand{\BigSet}[2]{\Big\{#1\mathbin{\Big|}#2\Big\}}

\newcommand{\eval}[2]{\left\langle#1,#2\right\rangle}     

\newcommand{\abs}[1]{\left\lvert#1\right\rvert}  
\newcommand{\smallabs}[1]{\lvert#1\rvert}
\newcommand{\bigabs}[1]{\big\lvert#1\big\rvert}
\newcommand{\Bigabs}[1]{\Big\lvert#1\Big\rvert}

\newcommand{\norm}[1]{\left\lVert#1\right\rVert} 

\newcommand{\bigrestrict}[2]{{#1\big|}_{#2}}

\newcommand{\compose}{\mathbin{\circ}}          
\newcommand{\define}{\mathrel{\rm:=}}

\renewcommand{\implies}{\mathrel{\Rightarrow}}

\newcommand{\AND}{\mathrel{\wedge}}

\newcommand{\without}{\setminus}
\newcommand{\blank}{\text{\textvisiblespace}}

\newcommand{\leer}{\emptyset}         

\newcommand{\R}{\mathbb{R}}           
\newcommand{\N}{\mathbb{N}}           
\newcommand{\Z}{\mathbb{Z}}           

\newcommand{\mfbd}{\partial}          
\DeclareMathOperator{\diag}{\mt{diag}}         
\DeclareMathOperator{\Gr}{\mt{Gr}}             

\newcommand{\ooi}[2]{\mathord{\left]#1,#2\right[}}  
\newcommand{\oci}[2]{\mathord{\left]#1,#2\right]}}  
\newcommand{\coi}[2]{\mathord{\left[#1,#2\right[}}  
\newcommand{\cci}[2]{\mathord{\left[#1,#2\right]}}

\newcommand{\smallcoi}[2]{\mathord{[#1,#2[}}  

\newcommand{\bigcci}[2]{\mathord{\big[#1,#2\big]}}

\newcommand{\smatrix}[1]{\left(\begin{smallmatrix} #1 \end{smallmatrix}\right)}

\renewcommand{\ln}{\operatorname{\mt{ln}}}

\renewcommand{\exp}{\operatorname{\mt{exp}}}
\renewcommand{\max}{\operatorname*{\mt{max}}}
\renewcommand{\min}{\operatorname*{\mt{min}}}
\renewcommand{\lim}{\operatorname*{\mt{lim}}}
\renewcommand{\sup}{\operatorname*{\mt{sup}}}
\renewcommand{\inf}{\operatorname*{\mt{inf}}}

\renewcommand{\dim}{\operatorname{\mt{dim}}}

\DeclareMathOperator{\id}{\mt{id}}             
\DeclareMathOperator{\pr}{\mt{pr}}             
\newcommand{\const}{\mathord{\mt{const}}}      
\DeclareMathOperator{\Sym}{\mt{Sym}}
\DeclareMathOperator{\dom}{\mt{dom}}
\DeclareMathOperator{\End}{\mt{End}}
\newcommand{\kuno}{\varowedge}

\renewcommand{\sec}{\operatorname{\mt{sec}}}
\DeclareMathOperator{\Riem}{\mt{Riem}}

\DeclareMathOperator{\grad}{\mt{grad}}
\DeclareMathOperator{\Hess}{\mt{Hess}}

\newcommand{\SecondFF}{\operatorname{\mathit{II}}}    

\DeclareMathOperator{\dist}{\mt{dist}}

\DeclareMathOperator{\length}{\mt{length}}

\newcommand{\eucl}{\mt{eucl}}

\newcommand{\jet}{\mt{j}}  

\DeclareMathOperator{\scal}{\mt{scal}}

\DeclareMathOperator{\conv}{\mt{conv}}
\DeclareMathOperator{\conj}{\mt{conj}}
\DeclareMathOperator{\inj}{\mt{inj}}
\newcommand{\Sph}{\mathbb{S}}
\newcommand{\T}{\mathbb{T}}
\DeclareMathOperator{\inte}{\mt{interior}}

\newcommand{\Eseq}{\mathcal{E}}
\newcommand{\Kex}{\mathcal{K}}
\newcommand{\Fol}{{\mathcal{F}}}
\newcommand{\Tens}{\mathord{\mathcal{T}}}
\newcommand{\PC}{\mathord{\mt{PC}}}
\newcommand{\Poly}[2]{\mathord{\R\mt{Poly}_{#1}^{#2}}}
\DeclareMathOperator{\Climbers}{\textsl{Climbers}}
\DeclareMathOperator{\Fct}{\mt{Fct}}
\newcommand{\Wick}{\mt{Wick}}

\begin{document}

\title{Every conformal class contains a metric of bounded geometry}

\author{Olaf M\"uller}
\address{Fakult\"at f\"ur Mathematik, Universit\"at Regensburg
}
\email{olaf.mueller@mathematik.uni-regensburg.de}

\author{Marc Nardmann}
\address{Fachbereich Mathematik, Universit\"at Hamburg}
\email{marc.nardmann@math.uni-hamburg.de}

\begin{abstract}
We show that on every manifold, every conformal class of semi-Riemannian metrics contains a metric $g$ such that each $k$th-order covariant derivative of the Riemann tensor of $g$ has bounded absolute value $a_k$. This result is new also in the Riemannian case, where one can arrange in addition that $g$ is complete with injectivity and convexity radius $\geq1$. One can even make the radii rapidly increasing and the functions $a_k$ rapidly decreasing at infinity. We prove generalizations to foliated manifolds, where curvature, second fundamental form and injectivity radius of the leaves can be controlled similarly. Still more generally, we introduce the notion of a \scare{flatzoomer}: a quantity that involves arbitrary geometric structures and behaves suitably with respect to modifications by a function, e.g.\ a conformal factor. The results on bounded geometry follow from a general theorem about flatzoomers, which might be applicable in many other geometric contexts involving noncompact manifolds.
\end{abstract}

\maketitle

\section{Introduction. Statement of results} \label{intro}

A classical result due to R.\ E.\ Greene \cite{Greene} says that every manifold admits a Riemannian metric of bounded geometry. It is therefore natural to ask a more refined question: Which conformal classes of Riemannian metrics on a given manifold contain metrics of bounded geometry? The question is of course trivial on compact manifolds, because every metric there has bounded geometry. The problem on open manifolds has been considered by Eichhorn--Fricke--Lang \cite{EFL}, who proved that certain quite special conformal classes on manifolds of suitable topology contain metrics of bounded geometry. In the present article, we will show that on every manifold, \emph{each} conformal class of Riemannian metrics contains a metric of bounded geometry. We also state and prove generalizations to foliated Riemannian manifolds and to semi-Riemannian manifolds of arbitrary signature, but let us first discuss the plain Riemannian case.

\begin{conventions} \label{conventions}
$0\in\N$. Manifolds are pure-dimensional, second countable, without boundary, and real-analytic. (Recall that the real-analyticity assumption is no loss of generality: For $r\in\N_{\geq1}\cup\set{\infty}$, every maximal $C^r$-atlas contains a real-analytic subatlas, and every two such subatlases are real-analytically diffeo\-morphic; cf.\ e.g.\ \cite{Shiga}.) Semi-Riemannian metrics and foliations are $C^\infty$. A manifold-with-boundary may have an empty boundary. A compact exhaustion of a manifold $M$ is a sequence $(K_i)_{i\in\N}$ of compact subsets of $M$ with $\bigcup_{i\in\N}K_i=M$ such that each $K_i$ is contained in the interior of $K_{i+1}$. A compact exhaustion $(K_i)_{i\in\N}$ is \textbf{smooth} iff all $K_i$ are $C^\infty$ co\-di\-men\-sion-$0$ submanifolds-with-boundary of $M$.
\end{conventions}

\begin{definition} \label{defbounded}
Let $M$ be a manifold, let $k\in\N$, let $\eps,\iota\in C^0(M,\R_{>0})$. A Riemannian metric $g$ on $M$ \textbf{has $k$-geometry bounded by $(\eps,\iota)$} iff
\begin{itemize}
\item $\bigabs{\nabla_g^i\Riem_g}_g \leq \eps$ holds for every $i\in\set{0,\dots,k}$; and
\item for each $x\in M$, the injectivity radius $\inj_g(x)\in\oci{0}{\infty}$ of $g$ at the point $x$ is $\geq\iota(x)$.
\end{itemize}
Here $\nabla_g^i\Riem_g$ denotes the $i$th covariant derivative with respect to $g$ of the Riemann tensor $\Riem_g$. (It does not matter whether we consider $\Riem_g$ as a $(4,0)$-tensor or $(3,1)$-tensor; the resulting functions $\bigabs{\nabla_g^i\Riem_g}_g \in C^0(M,\R_{\geq0})$ are the same in both cases.)

\smallskip
Let $\Kex = (K_i)_{i\in\N}$ be a compact exhaustion of $M$, let $\Eseq=(\eps_i)_{i\in\N}$ be a sequence in $C^0(M,\R_{>0})$. A Riemannian metric $g$ on $M$ \textbf{has $(\infty,\Kex)$-geometry bounded by $(\Eseq,\iota)$} iff
\begin{itemize}
\item for every $i\in\N$, the inequality $\bigabs{\nabla_g^i\Riem_g}_g \leq \eps_i$ holds on $M\without K_i$;
\item $\inj_g\geq\iota$.
\end{itemize}

According to standard terminology, a Riemannian metric $g$ on $M$ \textbf{has bounded geometry} iff there exist a sequence $\Eseq=(\eps_i)_{i\in\N}$ of positive \emph{constants} and a \emph{constant} $\iota\in\R_{>0}$ such that
\begin{itemize}
\item for every $i\in\N$, the inequality $\bigabs{\nabla_g^i\Riem_g}_g \leq \eps_i$ holds on $M$; and
\item $\inj_g \geq \iota$.
\end{itemize}
\end{definition}

For the relation of our \scare{\emph{$k$-geometry}} terminology to notions involving derivatives of the metric coefficients with respect to normal coordinates, see \cite{Eichhorn}.

\begin{factit} \label{equiv}
Let $g$ be a Riemannian metric on a manifold $M$. The following statements are equivalent:
\begin{enumerate}[label=(\arabic*)]
\item\label{st1} $g$ has bounded geometry.
\item\label{st2} There exist a compact exhaustion $\Kex$ of $M$, a sequence $\Eseq=(\eps_i)_{i\in\N}$ of positive constants, and a constant $\iota\in\R_{>0}$ such that $g$ has $(\infty,\Kex)$-geometry bounded by $(\Eseq,\iota)$.
\end{enumerate}
\end{factit}
\begin{proof}
\ref{st1}$\implies$\ref{st2} follows immediately from the fact that every manifold admits a compact exhaustion. \ref{st2}$\implies$\ref{st1} follows from the fact that a function on $M$ which is bounded on the complement of a compact set $K_i$ is bounded on $M$.
\end{proof}

Now we can state our main result for Rie\-mann\-ian metrics:

\begin{theorem} \label{main}
Let $\Kex=(K_i)_{i\in\N}$ be a smooth compact exhaustion of a manifold $M$, let $\iota,u_0\in C^0(M,\R_{>0})$, let $\Eseq$ be a sequence in $C^0(M,\R_{>0})$, let $g_0$ be a Riemannian metric on $M$. Then there exists a real-analytic function $u\colon M\to\R$ with $u>u_0$ such that the metric $\e^{2u}g_0$ is complete and has $(\infty,\Kex)$-geometry bounded by $(\Eseq,\iota)$.
\end{theorem}

The statement that the conformal class of $g_0$ contains a metric with $(\infty,\Kex)$-geometry bounded by $(\Eseq,\iota)$ becomes of course the stronger the more rapidly the elements of $\Eseq$ decay at infinity and the more rapidly $\iota$ increases at infinity.

\begin{corollary} \label{maincor}
Let $M$ be a manifold. Every conformal class of Riemannian metrics on $M$ contains a metric of bounded geometry. Every conformal class of Riemannian metrics on $M$ that contains a real-analytic metric contains a real-analytic metric of bounded geometry.
\end{corollary}

\begin{proof}
We choose a smooth compact exhaustion $\Kex$ of $M$, a sequence $\Eseq$ of positive constants, and a constant $\iota>0$. We apply Theorem \ref{main} to a metric $g_0$ --- a real-analytic one if possible --- in the given conformal class. The resulting $g=\e^{2u}g_0$ satisfies \ref{st2} from Fact \ref{equiv} and thus has bounded geometry.
\end{proof}

\noindent\emph{Remark 1.} Every manifold admits a real-analytic Riemannian metric by the Morrey--Grauert embedding theorem; cf.\ \cite{Shiga} and the references therein. But not every conformal class of Riemannian metrics contains a real-analytic one. For instance, on every nonempty manifold of dimension $\geq4$ one can easily construct a metric whose Weyl tensor is not real-analytic.

\medskip\noindent
\emph{Remark 2.} In the introduction to their article \cite{EFL}, Eichhorn--Fricke--Lang state in passing that it be easy to endow $\R^n$ with a metric which is not conformally equivalent to any metric of bounded geometry. Corollary \ref{maincor} disproves that.

\begin{corollary} \label{kbounded}
Let $k\in\N$, let $g_0$ be a Riemannian metric on a manifold $M$, let $\eps,\iota,u_0\in C^0(M,\R_{>0})$. Then there exists a real-analytic $u\colon M\to\R$ with $u>u_0$ such that $\e^{2u}g_0$ has $k$-geometry bounded by $(\eps,\iota)$.
\end{corollary}
\begin{proof}
We choose a smooth compact exhaustion $\Kex=(K_i)_{i\in\N}$ of $M$ with $K_i=\leer$ for $i\leq k$. We define $\Eseq$ to be the sequence all of whose entries are $\eps$. Theorem \ref{main} applied to $\Kex,\Eseq,\iota$ proves the claim.
\end{proof}

\noindent
\emph{Remark.} As stated in \ref{conventions}, we assume metrics to be $C^\infty$ for simplicity. Regularity $C^{k+2}$ would suffice for the Corollary \ref{kbounded} on $k$-bounded geometry, though, as interested readers will have no difficulty to check.

\begin{remark}
Since the standard definition of bounded geometry involves the injectivity radius, we have used it in the statements above. Replacing $\inj_g$ by the convexity radius $\conv_g$ in Definition \ref{defbounded} yields superficially stronger statements \ref{main}, \ref{maincor}, \ref{kbounded}, though, because every Riemannian metric $g$ satisfies $\conv_g\leq\inj_g$ (and even $2\conv_g\leq\inj_g$ holds for complete metrics $g$). However, \ref{equiv}, \ref{main}, \ref{maincor}, \ref{kbounded} remain true with $\conv_g$ instead of $\inj_g$, as we state explicitly in Theorem \ref{maingeneral} and prove in Section \ref{radii}.
\end{remark}

The proof of Theorem \ref{main} is similar to Greene's construction of metrics of bounded geometry \cite{Greene} in several aspects: Like us, Greene uses a compact exhaustion $(K_i)_{i\in\N}$, thereby decomposing $M$ into cylinders $Z_i$ diffeomorphic to $\R\times\mfbd K_i$ and \scare{topology-changing} regions $U_i$; and like us, he modifies a start metric $g_0$ only conformally, the conformal factor being constant on each $U_i$. The extreme simplification compared to our situation occurs on each of the sets $Z_i$, where Greene can choose $g_0$ to be a product metric, namely a very long cylinder, the length depending on the $g_0$-geometry of the neighboring regions $U_i$ and $U_{i+1}$. The only information he needs is that the functions $\bigabs{\nabla_g^i\Riem_g}{}_{g}$ and $\inj_g^{-1}$ become small when $g$ is multiplied by a large constant, and that they depend continuously on $g$ with respect to the compact-open $C^\infty$-topology. As we are not free to choose $g_0$, we have to work harder in the proof of \ref{main}, both with respect to $\bigabs{\nabla_g^i\Riem_g}{}_{g}$ and with respect to $\inj_g$. The non-obvious steps of the proof are Theorem \ref{RiemannInj} about injectivity radii and the choice of the cutoff functions in Lemma \ref{alpinist}. The rest of the argument consists of the conceptual setup and technical parts of the proof.

\medskip
One might ask whether Theorem \ref{main} could be improved with respect to extensions of metrics. For instance, \ref{main} says that for every $\eps\in C^0(M,\R_{>0})$, every conformal class of Riemannian metrics on $M$ contains a metric $g$ with $\abs{\Riem_g}{}_{\!g} < \eps$. In Gromov's h-principle language \cite{EM,GromovPDR}, this means that a certain (open second-order) partial differential relation for functions $M\to\R$ satisfies the h-principle. Whenever something like that happens, one should ask whether the relation satisfies even an \emph{h-principle for extensions}. In our case, the question is this: Given a closed subset $A$ of a manifold $M$ and a function $\eps\in C^0(M,\R_{>0})$, is the following statement true?
\begin{quote}
\scare{\emph{Let $g_0$ be a Riemannian metric on $M$ that fulfills $\abs{\Riem_{g_0}}{}_{\!g_0} < \eps$ on $A$. Then there exists a function $u\in C^\infty(M,\R)$ with $\bigrestrict{u}{A}=0$ such that $g\define\e^{2u}g_0$ fulfills $\abs{\Riem_g}{}_{\!g} < \eps$ on $M$.}}
\end{quote}
One can ask analogous questions for weaker relations like $\scal_g>-n(n-1)\eps$ instead of $\abs{\Riem_g}{}_{\!g}<\eps$, where $n=\dim M$. One can also weaken the statement:
\begin{quote}
\scare{\emph{Let $g_0$ be a Riemannian metric on $M$ that fulfills $\Riem_{g_0}=0$ on $A$. Then there exists a function $u\in C^\infty(M,\R)$ with $\bigrestrict{u}{A}=0$ such that $g\define\e^{2u}g_0$ fulfills $\scal_g > -\eps$ on $M$.}}
\end{quote}
Even this second statement is false: On every manifold $M$ of dimension $n\geq3$, for every compact codimension-$0$ submanifold-with-boundary $A$ of $M$ having $\mfbd A\neq\leer$, for every $\eps\in C^0(M,\R_{>0})$, and for every given Riemannian metric $\tilde{g}_0$ on $M$ that satisfies $\Riem_{\tilde{g}_0}=0$ on $A$, there exists a counterexample $g_0$ to the statement which is equal to $\tilde{g}_0$ on $A$ \cite{MN}. Hence the h-principle for extensions fails completely here.

This means that the differential relations we consider in the present article --- $\abs{\Riem_g}{}_{\!g} < \const\in\R_{>0}$, for instance --- belong to a type which is rare among strict partial differential inequalities arising naturally in geometry: they are flexible enough to admit solutions on manifolds of arbitrary topology, but the reason for this flexibility is \emph{not} that arbitrary solutions given on suitable closed subsets of $M$ could be extended to $M$. In contrast, when we drop the restriction to a given conformal class, then relations like $\abs{\Riem_g}{}_{\!g} < \const\in\R_{>0}$ \emph{are} flexible in the following strong sense: When $A$ is a closed subset of $M$ such that no connected component of $M\without A$ is relatively compact in $M$, then for every Riemannian metric $g_0$ on $M$ which satisfies the relation on $A$, there exists a (possibly not complete) metric $g$ that satisfies the relation on $M$ and is equal to $g_0$ on $A$. (This is a consequence of \cite[Theorem 7.2.4]{EM}; cf.\ \cite{GN} for details and generalizations.)

\medskip
Now that we have seen that Theorem \ref{main} is the best result one can hope for in the \scare{plain Riemannian} setting, let us discuss the announced generalizations to foliated manifolds and semi-Riemannian metrics. The core of our proof of Theorem \ref{main} is the construction of solutions to certain ordinary differential inequalities. This core argument does not involve any geometry. In Section \ref{flatzoomers}, we will axiomatize the general situation it applies to by introducing the notion of a \emph{flatzoomer}: a functional that assigns to functions $u\in C^\infty(M,\R)$ --- which in our context describe conformal factors --- functions $\Phi(u)\in C^0(M,\R_{\geq0})$ that satisfy certain estimates. For instance, for $i\in\N$ and a Riemannian metric $g$ on $M$, the functional $\Phi_i\colon u\mapsto \bigabs{\nabla_{g[u]}^i\Riem_{g[u]}}_{g[u]}$ with $g[u]\define\e^{2u}g$ is a flatzoomer. Leaving some subtleties of the injectivity radius aside, Theorem \ref{main} is obtained as a special case of a result about sequences of flatzoomers (e.g.\ the sequence $(\Phi_i)_{i\in\N}$), namely Theorem \ref{FlatzoomThemAll} below. More generally, one can use this abstract result to prove the following theorem.

\begin{theorem} \label{maingeneral}
Let $\Fol$ be a foliation on a manifold $M$, let $g_0,h_0$ be semi-Riemannian metrics (not necessarily of the same signature) on $M$ which induce (nondegenerate) semi-Riemannian metrics on (the leaves of) $\Fol$. Let $(K_i)_{i\in\N}$ be a smooth compact exhaustion of $M$, let $\iota,u_0\in C^0(M,\R_{>0})$, let $(\eps_i)_{i\in\N}$ be a sequence in $C^0(M,\R_{>0})$. Then there exists a real-analytic function $u\colon M\to\R$ with $u>u_0$ such that the metrics $g\define\e^{2u}g_0$ and $h\define\e^{2u}h_0$ have the following properties:
\begin{enumerate}[label=(\roman*)]
\item\label{gen1} For every $i\in\N$, $\bigabs{\nabla_g^i\Riem_g}_h < \eps_i$ holds on $M\without K_i$.
\item\label{gen2} If $g_0$ is Riemannian, then $g$ is complete; its injectivity and convexity radii satisfy $\inj_g\geq 2\conv_g > \iota$.
\item\label{gen3} For every $i\in\N$, $\bigabs{\nabla_{g_\Fol}^i\Riem_{g_\Fol}}_{h_\Fol} < \eps_i$ holds on $M\without K_i$.
\item\label{gen4} If $(g_0)_\Fol$ is Riemannian, then for each $\Fol$-leaf $L$, $g_L$ is complete with $\inj_{g_L} \geq 2\conv_{g_L} > \bigrestrict{\iota}{L}$.
\item\label{gen5} For every $i\in\N$, $\bigabs{\nabla_g^i\SecondFF^\Fol_g}_h < \eps_i$ holds on $M\without K_i$.
\end{enumerate}
\end{theorem}
Here $\nabla_g^i\SecondFF^\Fol_g$ denotes the $i$th covariant derivative with respect to $g$ of the second fundamental form of $\Fol$ with respect to $g$ (cf.\ \ref{ex4} for details); for a semi-Riemannian metric $\eta$ on $M$, $\eta_\Fol$ denotes the field of bilinear forms induced by $\eta$ on (the leaves of) $\Fol$; for every leaf $L$ of $\Fol$, $g_L$ denotes the metric on $L$ induced by $g$; $\nabla_{g_\Fol}^i\Riem_{g_\Fol}$ is the tensor field on $\Fol$ which assigns to each $x\in M$ the value of $\nabla_{g_L}^i\Riem_{g_L}$ at $x$, where $L$ is the leaf through $x$ (cf.\ \ref{ex3}); and the absolute value $\abs{T}_\eta$ of a tensor field $T$ with respect to a semi-Riemannian metric $\eta$ is defined to be the function $\abs{\eval{T}{T}_\eta}{}^{1/2}$ (cf.\ \ref{notation}).

\medskip\noindent
\emph{Remark 1.} Theorem \ref{main} is the special case of \ref{maingeneral} where $g_0=h_0$ is Riemannian.

\medskip\noindent
\emph{Remark 2.} The Riemannianness assumptions in Theorem \ref{maingeneral}\ref{gen2},\ref{gen4} cannot be avoided in general, as we discuss briefly in Section \ref{radii} below. In particular, not every conformal class of Lorentzian metrics contains a geodesically complete one, even on closed manifolds.

\medskip
The information that the metric $g$ we get from Theorem \ref{maingeneral} lies in a given conformal class is particularly important for indefinite metrics: then the causal structure (which is an invariant of the conformal class) plays a crucial role in many considerations. For instance, if the given $g_0$ is a globally hyperbolic or stably causal Lorentzian metric, then the metric $g$ provided by \ref{maingeneral} has the same property.

A typical special case of Theorem \ref{maingeneral} is the following Corollary \ref{mainlorentz} about stably causal Lorentzian metrics. Since every globally hyperbolic metric is stably causal, \ref{mainlorentz} applies in particular to the globally hyperbolic setting. Recall that a Lorentzian manifold $(M,g)$ is stably causal if and only if it admits a temporal function \cite{BS}, that is, a function $t\in C^\infty(M,\R)$ whose gradient $\grad_gt$ is $g$-timelike: $g(\grad_gt,\grad_gt)<0$. In this situation, we consider the Wick rotation of $g$ around the timelike subbundle $\R\grad_gt$ of $TM$; i.e., the Riemannian metric $\Wick(g,t)$ on $M$ defined by
\[
\Wick(g,t)(v,w) \define g(v,w) -\frac{2g(v,\grad_gt)g(w,\grad_gt)}{g(\grad_gt,\grad_gt)}
= g(v,w) -\frac{2\diff t(v)\diff t(w)}{\eval{\diff t}{\diff t}_g} \,.
\]
(On each level set of $t$, $g$ and $\Wick(g,t)$ induce the same Riemannian metric; whereas for $X\define\grad_gt$, we have $g(X,X)=-\Wick(g,t)(X,X)$.)

\begin{corollary} \label{mainlorentz}
Let $(M,g_0)$ be a Lorentzian manifold, let $t\in C^\infty(M,\R)$ have timelike $g_0$-gradient, let $\Fol$ denote the foliation of $M$ whose leaves are the level sets of $t$. Let $(K_i)_{i\in\N}$ be a smooth compact exhaustion of $M$, let $\iota,u_0\in C^0(M,\R_{>0})$, let $(\eps_i)_{i\in\N}$ be a sequence in $C^0(M,\R_{>0})$. Then there exists a real-analytic function $u\colon M\to\R$ with $u>u_0$ such that $g\define\e^{2u}g_0$ has the following properties:
\begin{enumerate}[label=(\roman*)]
\item\label{gen1b} For every $i\in\N$, $\bigabs{\nabla_g^i\Riem_g}_{\Wick(g,t)} < \eps_i$ holds on $M\without K_i$.
\item For every level set $L$ of $t$, $g_L$ is a complete Riemannian metric on $L$ with $\inj_{g_L} \geq 2\conv_{g_L} > \bigrestrict{\iota}{L}$.
\item\label{gen3b} For every $i\in\N$ and every level set $L$ of $t$, $\bigabs{\nabla_{g_L}^i\Riem_{g_L}}_{g_L} < \eps_i$ holds on $M\without K_i$.
\item\label{gen5b} For every $i\in\N$, $\bigabs{\nabla_g^i\SecondFF^\Fol_g}_{\Wick(g,t)} < \eps_i$ holds on $M\without K_i$.
\end{enumerate}
\end{corollary}

\noindent
\emph{Remark 1.} If one is only interested in estimates of finitely many derivatives of $\Riem_g$, $\Riem_{g_L}$ and $\SecondFF^\Fol_g$, one can choose a compact exhaustion whose first $N$ elements are empty (as in the proof of \ref{kbounded}) and thus gets estimates in \ref{gen1b}, \ref{gen3b}, \ref{gen5b} that hold on all of $M$.

\medskip\noindent
\emph{Remark 2.} In the situation of \ref{mainlorentz}, one wants to apply \ref{maingeneral} to a \emph{Riemannian} metric $h$ (instead of taking for instance $h=g$) in order to get sharper estimates; see the Remark after Example \ref{ex2} below. The metric $h=\Wick(g,t)$ is just the most natural choice.

\bigskip
Let us consider the case where $g_0=h_0$ is Riemannian in Theorem \ref{maingeneral}. Even if one is not interested in having a solution metric $g$ in each conformal class, the conformal class construction is probably the only chance to prove, for an \emph{arbitrary} foliation $\Fol$ and any given $\eps\in C^0(M,\R_{>0})$, the existence of a metric $g$ satisfying e.g.\ $\abs{\Riem_{g_\Fol}}_{g_\Fol}<\eps$ and $\bigabs{\SecondFF_g^\Fol}_g < \eps$. Since the foliation $\Fol$ will usually not fit to the structure of \emph{any} compact exhaustion $(K_i)_{i\in\N}$ of $M$ (in the sense that the boundaries $\mfbd K_i$ are not leaves of $\Fol$), a Greene-style construction would not work, for instance. The problem becomes even more severe when $g_0$ or $h_0$ is not Riemannian.

\medskip
Our method of proof, in particular Theorem \ref{FlatzoomThemAll}, should be regarded as a construction kit for all kinds of theorems in the spirit of \ref{maingeneral}. Instead of a foliation, such theorems might involve other geometric objects, e.g.\ bundles, almost complex structures or symplectic forms. Functions built from a metric $g$ and from such objects will often define flatzoomers via conformal change of $g$; cf.\ Remark \ref{zoomremark}. The flatzoomer condition is always easy to check for a given example. Whenever it holds, one gets a theorem of the form \ref{maingeneral} saying that the considered function is small for some (complete) metric $g$ in the desired conformal class.

Since bounded geometry entails nice analytic properties --- in particular Sobolev embeddings: cf.\ \cite[\S1.3]{EichhornGlobAna} ---, Theorem \ref{maingeneral} should be useful in several contexts: In conformally invariant field theories on curved Lorentzian or Riemannian backgrounds, it allows to choose a convenient background metric without leaving the conformal class. In the context of conformal-class-preserving flows like the Yamabe flow for Riemannian metrics on open manifolds, it could be used to choose a suitable start metric, which might then flow to a special metric within the given conformal class.

\medskip
The article is organized as follows: Flatzoomers are introduced in Section \ref{flatzoomers}. Injectivity and convexity radii are discussed in Section \ref{radii}. Section \ref{mainproof} contains the proofs of our main results.

\subsection*{Acknowledgments.} We thank Nadine Gro{\ss}e for pointing out reference \cite{EFL}; Stefan Suhr for remarks on the completeness of Lorentzian metrics; and Hans-Bert Rademacher for pointing out a mistake in an earlier version of Lemma \ref{babyKling}.

\section{Flatzoomers} \label{flatzoomers}

In this section, we introduce the notions \emph{flatzoomer} and \emph{quasi-flatzoomer} and give several examples.

\begin{notation}[{\protect $g[u]$, $\nabla_g^iT$, $\abs{T}{}_g$, $\Poly{m}{d}$}] \label{notation}
Let $(M,g)$ be a semi-Riemannian manifold.
\begin{itemize}
\item For $u\in C^\infty(M,\R)$, we denote the semi-Riemannian metric $\e^{2u}g$ by $g[u]$.
\item For $i\in\N$, the $i$th covariant derivative with respect to $g$ of a $C^\infty$ tensor field $T$ on $M$ is denoted by $\nabla_g^iT$.
\item The function $\eval{T}{T}_g\in C^\infty(M,\R)$ is the total contraction of $T\otimes T$ via $g$ in corresponding tensor indices. If $T$ is for instance a field of $k$-multilinear forms, this means that for every $x\in M$ and every $g$-orthonormal basis $(e_1,\dots,e_n)$ of $T_xM$ (where $g$-orthonormality means $g(e_i,e_j)=\eps_i\delta_{ij}$ for some $\eps_1,\dots,\eps_n\in\set{1,-1}$), we have
    \[
    \eval{T}{T}_g(x) = \sum_{i_1=1}^n\dots\sum_{i_k=1}^n\eps_{i_1}\cdots\eps_{i_k}T(e_{i_1},\dots,e_{i_k})^2 \,.
    \]
    Note that $\eval{T}{T}_g$ is not necessarily nonnegative if $g$ is not Riemannian.
\item The function $\abs{T}_g\in C^0(M,\R_{\geq0})$ is defined to be $\abs{\eval{T}{T}_g}{}^{\!1/2}$.
\item $\Riem_g$ denotes the Riemann tensor, viewed as a tensor field of type $(4,0)$. We adopt the Besse sign convention for $\Riem_g$ \cite{Besse}.
\item For $m,d\in\N$, $\Poly{m}{d}$ denotes the (finite-dimensional) $\R$-vector space of real polynomials of degree $\leq d$ in $m$ variables, equipped with its unique Hilbert space topology.
\end{itemize}
\end{notation}

\noindent
\emph{Remark.} Recall that $\eval{T}{T}_g$ does not change when we raise or lower indices of $T$ via $g$. In particular, functions like $\bigabs{\nabla_g^i\Riem_g}{}_g$ do not depend on whether one considers $\Riem_g$ as a $(4,0)$- or $(3,1)$-tensor field. However, when $h$ is another semi-Riemannian metric on $M$, then in general $\bigabs{\nabla_g^i\Riem_g}{}_h$ depends on this choice. The difference would not matter anywhere in this article, though, because the crucial flatzoomer properties would not be affected; see \ref{zoomremark} for related remarks.

\begin{definition} \label{fzdef}
Let $M$ be a manifold. A functional $\Phi\colon C^\infty(M,\R)\to C^0(M,\R_{\geq0})$ is a \textbf{flatzoomer} iff for some --- and hence every --- Riemannian metric $\eta$ on $M$, there exist $k,d\in\N$, $\alpha\in\R_{>0}$, $u_0\in C^0(M,\R)$ and a polynomial-valued map $P\in C^0\big(M,\Poly{k+1}{d}\big)$ such that
\[
\Phi(u)(x) \;\leq\; \e^{-\alpha u(x)}P(x)\left(u(x),\, \bigabs{\nabla_\eta^1u}_\eta(x),\, \dots,\, \bigabs{\nabla_\eta^ku}_\eta(x)\right)
\]
holds for all $x\in M$ and all $u\in C^\infty(M,\R)$ which satisfy $u(x)>u_0(x)$.
\end{definition}
\begin{proof}[Proof of \scare{and hence every}]
This is essentially straightforward and similar to but simpler than the proof of Example \ref{ex3} below. We omit the details.
\end{proof}

\begin{example}[covariant derivatives of the Riemann tensor] \label{ex1}
Let $(M,g)$ be a semi-Riemannian manifold, let $k\in\N$. Then $\Phi\colon C^\infty(M,\R)\to C^0(M,\R_{\geq0})$ defined by
\[
\Phi(u) \define \abs{\nabla^k_{g[u]}\Riem_{g[u]}}_{g[u]}
\]
is a flatzoomer. (This will be proved after Example \ref{ex3} below.)
\end{example}

\noindent
\emph{Remark.} Examples of this form, where $\Phi$ results from a given function like $\bigabs{\nabla^k_g\Riem_g}{}_g$ by varying $g$ conformally, motivate the terminology \scare{\emph{flatzoomer}}: As the \scare{zoom factor} $u$ becomes larger, $\Phi(u)$ becomes smaller (because $\e^{-\alpha u}$ tends to $0$) locally uniformly, provided the derivatives of $u$ are bounded in a suitable way described by $P$ and $\eta$. For instance, the curvature of $g[u]$ becomes smaller in the sense that $\abs{\Riem_{g[u]}}{}_{g[u]}$ tends to $0$; i.e., $g[u]$ becomes flatter.

\bigskip
We can generalize Example \ref{ex1}:
\begin{example}[covariant derivatives of the Riemann tensor again] \label{ex2}
Let $g,h$ be semi-Riemannian metrics (not necessarily of the same signature) on a manifold $M$, let $k\in\N$. Then $\Phi\colon C^\infty(M,\R)\to C^0(M,\R_{\geq0})$ defined by
\[
\Phi(u) \define \abs{\nabla^k_{g[u]}\Riem_{g[u]}}_{h[u]}
\]
is a flatzoomer. (This will be proved after Example \ref{ex3} below.)
\end{example}

\noindent
\emph{Remark.} Especially interesting is the case where $g$ is Lorentzian and $h$ is Riemannian. There are many situations, in particular in General Relativity, where one would like to have a Lorentzian metric $g$ on a manifold $M$ which makes a certain codimension-$1$ foliation $\Fol$ on $M$ spacelike, such that the curvature of $g$ is controlled in a stronger sense than $\abs{\Riem_g}{}_{\!g}$ being small: Typically, one wants to control certain components $\Riem_g(e_i,e_j,e_k,e_l)$ of $\Riem_g$, where $(e_0,\dots,e_{n-1})$ is a local orthonormal frame such that $e_1,\dots,e_{n-1}$ are tangential to the spacelike foliation $\Fol$ (and thus $e_0$ is timelike). However, the terms $\Riem_g(e_i,e_j,e_k,e_l)^2$ occur with different signs in the sum $\eval{\Riem_g}{\Riem_g}{}_{\!g}$. Thus the condition of $\abs{\Riem_g}{}_{\!g}$ being small is too weak; one wants that $\abs{\Riem_g}{}_{\!h}$ is small for some \emph{Riemannian} metric $h$. (When $\Fol$ is already given, it is natural to take the $h$ which one obtains from $g$ by changing the sign in the direction orthogonal to $\Fol$, as in Theorem \ref{mainlorentz}. Example \ref{ex2} works with an arbitrary $h$, though.)

\bigskip
Even more generally than Example \ref{ex2} (in the sense that \ref{ex2} results from considering the codimension-$0$ foliation whose only leaf is $M$), we can consider the curvature of the leaves of a foliation on $M$ instead of the curvature of the whole manifold $M$:

\begin{example}[covariant derivatives of the Riemann tensor of a foliation] \label{ex3}
Let $\Fol$ be a foliation on a manifold $M$, let $g,h$ be semi-Riemannian metrics on $\Fol$ (i.e., $g,h$ are smooth sections in the bundle $\Sym^2_\mt{nd}T^\ast\Fol\to M$, whose fiber over $x$ consists of the nondegenerate symmetric bilinear forms on the tangent space $T_x\Fol$ of the $\Fol$-leaf that contains $x$), let $k\in\N$. Generalizing \ref{notation}, we write $\tilde{g}[u]\define\e^{2u}\tilde{g}$ for any semi-Riemannian metric $\tilde{g}$ on $\Fol$; and similarly for the other notation in \ref{notation}. Then $\Phi\colon C^\infty(M,\R)\to C^0(M,\R_{\geq0})$ defined by
\[
\Phi(u) \define \abs{\nabla^k_{g[u]}\Riem_{g[u]}}_{h[u]}
\]
is a flatzoomer.
\end{example}

\noindent
\emph{Remark.} Note that the signature of $g$ resp.\ $h$ is automatically constant on each connected component of $M$, for continuity reasons.

\begin{proof}[Proof of \ref{ex3}, and thus of \ref{ex1} and \ref{ex2}]
For $r\in\N$, let $\Tens_r(\Fol)\to M$ denote the $\R$-vector bundle of $(r,0)$-tensors on $\Fol$; thus the fiber over $x\in M$ consists of the $r$-multilinear forms on $T_x\Fol$.

\smallskip
For $u\in C^\infty(M,\R)$, the $(4,0)$-Riemann curvature of $g[u]$ is \cite[Theorem 1.159b]{Besse}
\begin{equation} \label{confRiem}
\Riem_{g[u]} = \e^{2u}\Big(\Riem_{g} -g\kuno\Big(\Hess_{g}u -\dif u\otimes\dif u +\tfrac{1}{2}\abs{\dif u}_{g}^2\,g\Big)\Big) \,.
\end{equation}
With the notation $\nabla^{\tilde{g}} \equiv \nabla_{\tilde{g}}^1$, we have \cite[Theorem 1.159a]{Besse} for all sections $X$ in $T\Fol\to M$ and all $v\in T\Fol$:
\begin{equation} \label{confLC}
\nabla^{g[u]}_vX = \nabla^{g}_vX +\dif u(X)v +\dif u(v)X -g(v,X)\grad_{g}u \,.
\end{equation}

Let $k,m\in\N$. We consider the (finite-dimensional) $\R$-vector space $\PC^{g,\Fol}_{k,m}$ of base-preserving vector bundle morphisms $\Tens_{k+4+2m}(\Fol)\to\Tens_{k+4}(\Fol)$ which is spanned by all morphisms of the form $\xi\compose\pi$, where $\pi\colon\Tens_{k+4+2m}(\Fol)\to\Tens_{k+4+2m}(\Fol)$ is a permutation of tensor indices (the same permutation over each $x\in M$) and $\xi\colon\Tens_{k+4+2m}(\Fol)\to\Tens_{k+4}(\Fol)$ contracts each of the first $m$ pairs of indices via $g$.

\smallskip
We claim that for every $k\in\N$, there exist a number $\mu_k\in\N$ and, for each $i\in\set{1,\dots,\mu_k}$,
\begin{itemize}
\item a number $a_{k,i}\in\N$ and a section $\omega_{k,i}$ in $\Tens_{a_{k,i}}(\Fol)\to M$,
\item numbers $c_{k,i,1},\dots,c_{k,i,k+2}\in\N$,
\item a number $m_{k,i}\in\N$ with $a_{k,i} +\sum_{\nu=1}^{k+2}\nu c_{k,i,\nu} = k+4+2m_{k,i}$,
\item and a morphism $\psi_{k,i}\in\PC^{g,\Fol}_{k,m_{k,i}}$
\end{itemize}
such that the following equation holds for all $u\in C^\infty(M,\R)$:
\begin{equation} \label{Riemclaim} \begin{split}
\nabla^k_{g[u]}\Riem_{g[u]}
\;=\; \e^{2u}\;\sum_{i=1}^{\mu_k}\;\psi_{k,i}\bigg( \omega_{k,i}
\otimes \left(\nabla^1_{g}u\right)^{\otimes c_{k,i,1}}
\otimes\dots\otimes \left(\nabla^{k+2}_{g}u\right)^{\otimes c_{k,i,k+2}} \bigg) \,.
\end{split} \end{equation}

We prove this by induction over $k$. Equation \eqref{confRiem} shows that in the case $k=0$, \eqref{Riemclaim} holds with $\mu_0=4$,
\begin{align*}
a_{0,1} &= 4 \,,
&a_{0,2} &= 2 \,,
&a_{0,3} &= 2 \,,
&a_{0,4} &= 4 \,,\\
\omega_{0,1} &= \Riem_{g} \,,
&\omega_{0,2} &= g \,,
&\omega_{0,3} &= g \,,
&\omega_{0,4} &= g\kuno g \,,\\
c_{0,1,1} &= 0 \,,
&c_{0,2,1} &= 0 \,,
&c_{0,3,1} &= 2 \,,
&c_{0,4,1} &= 2 \,,\\
c_{0,1,2} &= 0 \,,
&c_{0,2,2} &= 1 \,,
&c_{0,3,2} &= 0 \,,
&c_{0,4,2} &= 0 \,,\\
m_{0,1} &= 0 \,,
&m_{0,2} &= 0 \,,
&m_{0,3} &= 0 \,,
&m_{0,4} &= 1 \,,
\end{align*}
for suitable morphisms $\psi_{0,1},\psi_{0,2},\psi_{0,3}\in\PC^{g,\Fol}_{0,0}$ and $\psi_{0,4}\in\PC^{g,\Fol}_{0,1}$.

\smallskip
Now we assume that \eqref{Riemclaim} holds for some $k\in\N$ and verify it for $k+1$. Since all elements of $\PC^{g,\Fol}_{k,\ast}$ are $g$-parallel, we obtain (using $\nabla^1_{g[u]}u = \dif u = \nabla^1_{g}u$ and the product and chain rules)
\[ \begin{split}
&\nabla_{g[u]}^{k+1}\Riem_{g[u]} = \nabla_{g[u]}\nabla_{g[u]}^k\Riem_{g[u]}\\
&= \e^{2u}\;\sum_{i=1}^{\mu_k}\;\nabla_{g[u]}\left( \psi_{k,i}\bigg( \omega_{k,i}
\otimes \left(\nabla^1_{g}u\right)^{\otimes c_{k,i,1}}
\otimes\dots\otimes \left(\nabla^{k+2}_{g}u\right)^{\otimes c_{k,i,k+2}} \bigg) \right)\\
&\mspace{16mu}+2\e^{2u}\diff u\otimes\sum_{i=1}^{\mu_k}\;\psi_{k,i}\bigg( \omega_{k,i}
\otimes \left(\nabla^1_{g}u\right)^{\otimes c_{k,i,1}}
\otimes\dots\otimes \left(\nabla^{k+2}_{g}u\right)^{\otimes c_{k,i,k+2}} \bigg)\\
&= \underbrace{\e^{2u}\;\sum_{i=1}^{\mu_k}\;\hat{\psi}_{k,i}\bigg( \nabla_{g[u]}\omega_{k,i}
\otimes \left(\nabla^1_{g}u\right)^{\otimes c_{k,i,1}}
\otimes\dots\otimes \left(\nabla^{k+2}_{g}u\right)^{\otimes c_{k,i,k+2}} \bigg)}_{\mt{I}\define}\\
&\mspace{16mu}+\underbrace{\e^{2u}\;\sum_{i=1}^{\mu_k}\sum_{j=1}^{k+2}\;\psi_{k,i,j}\bigg( \omega_{k,i}
\otimes \left(\nabla^1_{g}u\right)^{\otimes c_{k,i,1}} \otimes\dots\otimes
\nabla_{g[u]}\left(\left(\nabla^j_{g}u\right)^{\otimes c_{k,i,j}}\right)
\otimes\dots\otimes \left(\nabla^{k+2}_{g}u\right)^{\otimes c_{k,i,k+2}} \bigg)}_{\mt{II}\define}\\
&\mspace{16mu}+\underbrace{\e^{2u}\;\sum_{i=1}^{\mu_k}\;\tilde{\psi}_{k,i}\bigg( \omega_{k,i}
\otimes \left(\nabla^1_{g}u\right)^{\otimes c_{k,i,1}+1} \otimes \left(\nabla^2_{g}u\right)^{\otimes c_{k,i,2}}
\otimes\dots\otimes \left(\nabla^{k+2}_{g}u\right)^{\otimes c_{k,i,k+2}} \bigg)}_{\mt{III}\define}
\end{split} \]
for suitable $\hat{\psi}_{k,i},\; \psi_{k,i,j},\; \tilde{\psi}_{k,i}\in\PC^{g,\Fol}_{k+1,m_{k,i}}$. Summand III has already the desired form of the right-hand side of \eqref{Riemclaim}. Now we consider I. Writing $V_u(v,X)\define \dif u(X)v +\dif u(v)X -g(v,X)\grad_{g}u$ for $v,X\in T_xM$, we deduce from \eqref{confLC} (by applying the product rule twice):
\[ \begin{split}
&\left(\nabla^{g[u]}_v\omega_{k,i}\right)\left(v_1,\dots,v_{a_{k,i}}\right)
= \left(\nabla^{g}_v\omega_{k,i}\right)\left(v_1,\dots,v_{a_{k,i}}\right)
-\sum_{l=1}^{a_{k,i}}\omega_{k,i}\left(v_1,\dots,v_{l-1},V_u(v,v_l),v_{l+1},\dots,v_{a_{k,i}}\right)\\
&= \left(\nabla^{g}_v\omega_{k,i}\right)\left(v_1,\dots,v_{a_{k,i}}\right)
-\sum_{l=1}^{a_{k,i}}\omega_{k,i}\left(v_1,\dots,v_{l-1},v,v_{l+1},\dots,v_{a_{k,i}}\right)\nabla_{g}^1u(v_l)\\
&\mspace{16mu}-\sum_{l=1}^{a_{k,i}}\omega_{k,i}\left(v_1,\dots,v_{a_{k,i}}\right)\nabla_{g}^1u(v)
+\eval{\,{}\sum_{l=1}^{a_{k,i}}\omega_{k,i}\left(v_1,\dots,v_{l-1},\blank,v_{l+1},\dots,v_{a_{k,i}}\right)}{ \;g(v,v_l)\nabla_{g}^1u(\blank)}_{\!\!g} \,;
\end{split} \]
hence
\[ \begin{split}
&\hat{\psi}_{k,i}\bigg( \nabla_{g[u]}\omega_{k,i}
\otimes \left(\nabla^1_{g}u\right)^{\otimes c_{k,i,1}}
\otimes\dots\otimes \left(\nabla^{k+2}_{g}u\right)^{\otimes c_{k,i,k+2}} \bigg)\\
&= \hat{\psi}_{k,i}\bigg( \nabla_{g}\omega_{k,i}
\otimes \left(\nabla^1_{g}u\right)^{\otimes c_{k,i,1}}
\otimes\dots\otimes \left(\nabla^{k+2}_{g}u\right)^{\otimes c_{k,i,k+2}} \bigg)\\
&\mspace{16mu}+\sum_{l=1}^{a_{k,i}}\varphi_{k,i,l}\bigg( \omega_{k,i}
\otimes \left(\nabla^1_{g}u\right)^{\otimes c_{k,i,1}+1}
\otimes \left(\nabla^2_{g}u\right)^{\otimes c_{k,i,1}} \otimes\dots\otimes \left(\nabla^{k+2}_{g}u\right)^{\otimes c_{k,i,k+2}} \bigg)\\
&\mspace{16mu}+\sum_{l=1}^{a_{k,i}}\chi_{k,i,l}\bigg( g\otimes\omega_{k,i}
\otimes \left(\nabla^1_{g}u\right)^{\otimes c_{k,i,1}+1}
\otimes \left(\nabla^2_{g}u\right)^{\otimes c_{k,i,1}} \otimes\dots\otimes \left(\nabla^{k+2}_{g}u\right)^{\otimes c_{k,i,k+2}} \bigg)
\end{split} \]
for some $\varphi_{i,k,l}\in\PC^{g,\Fol}_{k+1,m_{k,i}}$ and $\chi_{i,k,l}\in\PC^{g,\Fol}_{k+1,m_{k,i}+1}$. This shows that also summand I has the desired form. A similar formula holds for each summand of
\[ \begin{split}
\nabla_{g[u]}\left(\left(\nabla^j_{g}u\right)^{\otimes c_{k,i,j}}\right)
= \sum_{\nu=1}^{c_{k,i,j}} \left(\nabla^j_{g}u\right)^{\otimes\nu-1} \otimes \nabla_{g[u]}\nabla^j_{g}u \otimes \left(\nabla^j_{g}u\right)^{\otimes c_{k,i,j}-\nu} \,,
\end{split} \]
which takes care of term II. Thus $\nabla_{g[u]}^{k+1}\Riem_{g[u]}$ has the required form \eqref{Riemclaim}. This completes the proof of our claim involving \eqref{Riemclaim}.

\smallskip
Let $u\in C^\infty(M,\R)$. To compute $\Phi(u)$ at a point $x\in M$, we choose an $h$-orthonormal basis $(e_1,\dots,e_n)$ of $T_x\Fol$. Then $(e_1[u],\dots,e_n[u])$ defined by $e_i[u]\define \e^{-u}e_i$ is an $h[u]$-orthonormal basis of $T_x\Fol$. Let $\eps_i\define h(e_i,e_i)\in\set{-1,1}$. Thus
\[ \begin{split}
\Phi(u) &= \abs{\nabla^k_{g[u]}\Riem_{g[u]}}_{h[u]}
= \abs{ \sum_{a\in\set{1,\dots,n}^{k+4}}\eps_{a_1}\dots\eps_{a_{k+4}} \left(\nabla^k_{g[u]}\Riem_{g[u]}\right)\left(e_{a_1}[u],\dots,e_{a_{k+4}}[u]\right)^2 }^{1/2}\\
&= \abs{ \e^{-2(k+4)u}\sum_{a\in\set{1,\dots,n}^{k+4}}\eps_{a_1}\dots\eps_{a_{k+4}} \left(\nabla^k_{g[u]}\Riem_{g[u]}\right)\left(e_{a_1},\dots,e_{a_{k+4}}\right)^2 }^{1/2}
= \e^{-(k+4)u}\abs{\nabla^k_{g[u]}\Riem_{g[u]}}_h \,.
\end{split} \]
Let $\eta$ be any Riemannian metric on $M$. For suitable $d\in\N$ and $P\in C^0(M,\Poly{k+2}{d})$ not depending on $u$, we obtain at every $x\in M$, using \eqref{Riemclaim},
\[ \begin{split}
\Phi(u)(x)
&= \e^{-(k+4)u(x)}\abs{\nabla^k_{g[u]}\Riem_{g[u]}}_{h}(x)\\
&= \e^{-(k+2)u(x)}\abs{ \sum_{i=1}^{\mu_k}\;\psi_{k,i}\bigg( \omega_{k,i}
\otimes \left(\nabla^1_{g}u\right)^{\otimes c_{k,i,1}}
\otimes\dots\otimes \left(\nabla^{k+2}_{g}u\right)^{\otimes c_{k,i,k+2}} \bigg) }_h(x)\\
&\leq \e^{-(k+2)u(x)}P(x)\left(\bigabs{\nabla_\eta^1u}_\eta(x),\dots,\bigabs{\nabla_\eta^{k+2}u}_\eta(x)\right) \,.
\end{split} \]
Hence $\Phi$ is a flatzoomer.
\end{proof}

\begin{example}[covariant derivatives of the second fundamental form of a foliation] \label{ex4}
Let $g,h$ be semi-Riemannian metrics on a manifold $M$, let $k\in\N$. Let $\Fol$ be a foliation on $M$ such that $g$ induces a semi-Riemannian metric $g_\Fol$ on the leaves of $\Fol$. (The condition that $g$ induces a semi-Riemannian metric on $\Fol$ is satisfied for instance when $g$ is Riemannian; more generally, when $\Fol$ is $g$-spacelike or $g$-timelike.) Let $\pr_g\colon TM\to T\Fol$ denote the $g$-orthogonal projection onto $T\Fol$; then $\pr_g^\bot\define\id_{TM}-\pr_g$ is pointwise the $g$-orthogonal projection from $T_xM$ onto the $g$-orthogonal complement of $T_x\Fol$ in $T_xM$. We consider the second fundamental form $\SecondFF^\Fol_g$ of $\Fol$ in $(M,g)$ as a field of trilinear forms on $M$; i.e., for all $x\in M$ and $v,w,z\in T_xM$, we let
\[
\SecondFF^\Fol_g(v,w,z) \define g\left(\nabla^g_{\pr_g(v)}\big(\pr_g\compose\hat{w}\big),\, \pr_g^\bot(z)\right) \,,
\]
where $\hat{w}$ is any vector field on $M$ with $\hat{w}(x)=w$ (the choice does not matter). Thus $\SecondFF^\Fol_g$ projects the input vectors $v,w\in T_xM$ to $T_x\Fol$, evaluates the second fundamental form of the $\Fol$-leaf through $x$ in these projections, and translates the resulting vector (which is normal to $T_x\Fol$) into a $1$-form.

Then $\Phi\colon C^\infty(M,\R)\to C^0(M,\R_{\geq0})$ defined by
\[
\Phi(u) \define \abs{\nabla^k_{g[u]}\SecondFF^\Fol_{g[u]}}_{h[u]}
\]
is a flatzoomer.
\end{example}
\begin{proof}
Let $u\in C^\infty(M,\R)$. Clearly $\pr\define\pr_g=\pr_{g[u]}$ and $\pr^\bot\define \pr_g^\bot=\pr_{g[u]}^\bot$. All $v\in TM$ and vector fields $X$ on $M$ satisfy \cite[Theorem 1.159a]{Besse}
\begin{equation} \label{confLC2}
\nabla^{g[u]}_vX = \nabla^g_vX +\dif u(X)v +\dif u(v)X -g(v,X)\grad_gu \,.
\end{equation}
This yields for all $x\in M$ and $v,w,z\in T_xM$:
\begin{equation} \label{confFF} \begin{split}
\SecondFF^\Fol_{g[u]}(v,w,z)
&= g[u]\left(\nabla^{g[u]}_{\pr(v)}\left(\pr\compose\hat{w}\right),\; \pr^\bot(z)\right)\\
&= g[u]\left(\nabla^g_{\pr(v)}\left(\pr\compose\hat{w}\right) -g\left(\pr(v),\pr(w)\right)\grad_gu,\; \pr^\bot(z)\right)\\
&= \e^{2u}\SecondFF^\Fol_g(v,w,z) -\e^{2u}g\left(\pr(v),\pr(w)\right)\, \dif u\left(\pr^\bot(z)\right) \,.
\end{split} \end{equation}

For $r\in\N$, we define $\Pi^g_r$ to be the set of sections in $\End(TM)^{\otimes r}\to M$ which have the form $p_1\otimes\dots\otimes p_r$ with $p_1,\dots,p_r\in\set{\pr,\,\pr^\bot,\,\id_{TM}}$. Using the notation $\Tens_r(M)$ and $\PC^{g,M}_{k,m}$ from the proof of Example \ref{ex3} (where $M$ stands for the foliation on $M$ whose only leaf is $M$), we claim that for every $k\in M$, there exist a number $\mu_k\in\N$ and, for each $i\in\set{1,\dots,\mu_k}$,
\begin{itemize}
\item a number $a_{k,i}\in\N$ and a section $\omega_{k,i}$ in $\Tens_{a_{k,i}}(M)\to M$,
\item numbers $c_{k,i,1},\dots,c_{k,i,k+1}\in\N$,
\item a number $m_{k,i}\in\N$ with $a_{k,i}+\sum_{\nu=1}^{k+1}\nu c_{k,i,\nu} = k+3+2m_{k,i}$,
\item a section $p_{k,i}\in\Pi^g_{k+3+2m_{k,i}}$ and a morphism $\psi_{k,i}\in\PC^{g,M}_{k-1,m_{k,i}}$
\end{itemize}
such that the following equation holds for all $u\in C^\infty(M,\R)$:
\begin{equation} \label{fundclaim}
\nabla^k_{g[u]}\SecondFF^\Fol_{g[u]} = \e^{2u}\sum_{i=1}^{\mu_k}\psi_{k,i}\bigg(\bigg(
\omega_{k,i}\otimes\left(\nabla_g^1u\right)^{\otimes c_{k,i,1}} \otimes\dots\otimes \left(\nabla_g^{k+1}u\right)^{\otimes c_{i,k,k+1}} \bigg)\compose p_{k,i}\bigg) \,.
\end{equation}
This claim is proved by induction over $k$ in a similar way as in the proof of Example \ref{ex3}, with \eqref{confFF} as induction start and \eqref{confLC2} being applied in the induction step. We omit the details.

Let $u\in C^\infty(M,\R)$. An estimate analogous to the end of the proof of Example \ref{ex3} yields now
\[
\forall x\in M\colon\; \Phi(u)(x)
\leq \e^{-(k+1)u(x)}P(x)\left(\bigabs{\nabla_\eta^1u}_\eta(x),\dots,\bigabs{\nabla_\eta^{k+1}u}_\eta(x)\right)
\]
for any Riemannian metric $\eta$ on $M$ and suitable $d\in\N$ and $P\in C^0(\Poly{k+1}{d})$ not depending on $u$. Hence $\Phi$ is a flatzoomer.
\end{proof}

\begin{remark} \label{zoomremark}
When we replace $h[u]$ by $h$ in the definitions of the respective maps $\Phi$ in the Examples \ref{ex2}, \ref{ex3}, \ref{ex4}, then these maps are no longer flatzoomers, as one can tell easily from the proofs above. In contrast, after replacing one or both of the symbols $g[u]$ by $g$ (while keeping $h[u]$) in one of the definitions, the resulting map $\Phi$ is still a flatzoomer. Replacing $\nabla_{g[u]}$ by $\nabla_{h[u]}$ or an arbitrary fixed connection $\tilde{\nabla}$ does not affect the flatzoomer property either.

As mentioned in Section \ref{intro}, when additional geometric objects --- e.g.\ an almost complex structure $J$ or a symplectic form --- are given on $M$, one can construct many other examples of flatzoomers. Up to some power --- e.g.\ $\frac{1}{2}$ in our examples above ---, these will typically be total $h[u]$-contractions $\Phi(u)$ of some tensor field $T_u$ built from the additional objects and $g[u]$ resp.\ $h[u]$; e.g., $T_u$ may be the $h[u]$-covariant derivative of the Nijenhuis tensor $N_J$, which has up to a sign exactly one not a priori vanishing total contraction. After lowering all upper indices of $T_u$ via $h[u]$, we may assume that $T_u$ is a field of multilinear forms. This $T_u$ will usually for some $c\in\Z$ have the form $\e^{cu}$ times a polynomial in $u$ and its derivatives; e.g., $c=4$ in Example \ref{ex3} with $T_u = \nabla^k\Riem_{g[u]}\otimes\nabla^k\Riem_{g[u]}$ (cf.\ \eqref{Riemclaim}), $c=4$ in Example \ref{ex4} with $T_u = \nabla^k\SecondFF_{g[u]}^\Fol\otimes\nabla^k\SecondFF_{g[u]}^\Fol$ (cf.\ \eqref{fundclaim}), and $c=2$ in the Nijenhuis derivative example. If the multilinear form $T_u$ has more than $c$ slots --- which is the case in all these examples ---, then the functional $\Phi$ is a flatzoomer.
\end{remark}

\begin{example} \label{ex5}
Let $M$ be a manifold, let $m\in\N$. For $i\in\set{1,\dots,m}$, let $\Phi_i\colon C^\infty(M,\R)\to C^0(M,\R_{\geq0})$ be a flatzoomer. Assume that $Q\in C^0\big(M\times(\R_{\geq0})^m,\,\R_{\geq0}\big)$ is homogeneous-polynomially bounded in the sense that there exist $r\in\R_{>0}$ and $c\in C^0(M,\R_{\geq0})$ with
\[
\forall x\in M\colon \forall v_1,\dots,v_m\in\cci{0}{1}\colon Q(x,v_1,\dots,v_m) \leq c(x)\cdot\big(v_1+\dots+v_m\big)^r \,.
\]
Then the functional $\Phi\colon C^\infty(M,\R)\to C^0(M,\R_{\geq0})$ defined by
\[
\Phi(u)(x) \define Q\big(x,\Phi_1(u)(x),\dots,\Phi_m(u)(x)\big)
\]
is a flatzoomer; cf. the proof sketch below.

This applies in particular to the function $Q$ given by $Q(x,v) = \sum_{i=1}^mv_i$. Thus $\Phi\define\sum_{i=1}^m\Phi_i$ is a flatzoomer. In this way, finitely many flatzoomers can be controlled by a single flatzoomer: if $\Phi(u)\leq\eps$ holds for some $u\in C^\infty(M,\R)$ and $\eps\in C^0(M,\R_{>0})$, then $\Phi_i(u)\leq\eps$ for every $i\in\set{1,\dots,m}$.

Another example is obtained by taking $m=1$ and $Q(x,s)=s^{1/2}$.
\end{example}
\begin{proof}[Sketch of proof of the flatzoomer property]
This is completely analogous to the proof of \ref{quasiex2} below: in the proof there, just replace every term of the form $\sup\Set{\,\mt{something}(y)}{y\in K_{l+1}\without K_{l-2}}$ by $\mt{something}(x)$; every \scare{$u>u_?$ on $K_{l+1}\without K_{l-2}$} by \scare{$u(x)>u_?(x)$}; and the last sentence by \scare{Thus $\Phi$ is a flatzoomer.}.
\end{proof}

In order to prove Theorem \ref{main}, we have to control not only the functions $\bigabs{\nabla_g^i\Riem_g}{}_g$ but also the inverse $\inj_g^{-1}\in C^0(M,\R_{>0})$ of the injectivity radius. However, the functional $\Phi\colon u\mapsto\inj_{g[u]}^{-1}$ is not a flatzoomer, because $\Phi(u)(x)$ cannot be bounded just in terms of some $k$-jet $\jet^k_xu$ of $u$ \emph{at the point $x$}; one has to take the values of $u$ on a whole neighborhood of $x$ into account. The following more general definition covers such functionals.

\smallskip
For a manifold $M$, let $\Fct(M,\R_{\geq0})$ denote the set of (not necessarily continuous) functions $M\to \R_{\geq0}$.

\begin{definition} \label{quasidef}
Let $\Kex=(K_i)_{i\in\N}$ be a compact exhaustion of a manifold $M$, let $K_{-2}\define K_{-1}\define\leer$. A functional $\Phi\colon C^\infty(M,\R)\to \Fct(M,\R_{\geq0})$ is a \textbf{quasi-flatzoomer for $\Kex$} iff for some --- and hence every --- Riemannian metric $\eta$ on $M$, there exist $k,d\in\N$, $\alpha\in\R_{>0}$, $u_0\in C^0(M,\R)$ and $P\in C^0\big(M,\Poly{k+1}{d}\big)$ such that
\[
\Phi(u)(x) \;\leq\; \sup\Set{\e^{-\alpha u(y)} P(y)\left(u(y),\, \bigabs{\nabla_\eta^1u}_\eta(y),\, \dots,\, \bigabs{\nabla_\eta^ku}_\eta(y)\right)}{y\in K_{i+1}\without K_{i-2}}
\]
holds for all $i\in\N$ and $x\in K_i\without K_{i-1}$ and $u\in C^\infty(M,\R)$ which satisfy $u>u_0$ on $K_{i+1}\without K_{i-2}$.
\end{definition}
\begin{proof}[Proof of \scare{and hence every}]
This is analogous to the proof of \ref{fzdef}.
\end{proof}

\begin{example} \label{quasiex1}
Every flatzoomer $\Phi\colon C^\infty(M,\R)\to C^0(M,\R_{\geq0})$ is a quasi-flatzoomer for every compact exhaustion of $M$.\qed
\end{example}

\begin{example} \label{quasiex2}
Let $\Kex=(K_l)_{l\in\N}$ be a compact exhaustion of a manifold $M$, let $m\in\N$. For $i\in\set{1,\dots,m}$, let $\Phi_i\colon C^\infty(M,\R)\to \Fct(M,\R_{\geq0})$ be a quasi-flatzoomer for $\Kex$. Assume $Q\in C^0\big(M\times(\R_{\geq0})^m,\,\R_{\geq0}\big)$ is homogeneous-polynomially bounded in the sense of Example \ref{ex5}. Then $\Phi\colon C^\infty(M,\R)\to \Fct(M,\R_{\geq0})$ defined by
\[
\Phi(u)(x) \define Q\big(x,\Phi_1(u)(x),\dots,\Phi_m(u)(x)\big)
\]
is a quasi-flatzoomer for $\Kex$.
\end{example}
\begin{proof}
Let $K_{-2}\define K_{-1}\define\leer$, let $\eta$ be a Riemannian metric on $M$. For each $i\in\set{1,\dots,m}$, there exist $k_i,d_i\in\N$, $\alpha_i\in\R_{>0}$ and $b_i,u_i\in C^0(M,\R_{\geq0})$ such that
\[ \begin{split}
\Phi_i(u)(x) &\;\leq\; \sup\Set{\e^{-\alpha_iu(y)}\,b_i(y)\cdot\left(1+\sum_{j=0}^{k_i}\abs{\nabla_\eta^ju}_\eta(y)\right)^{d_i} }{y\in K_{l+1}\without K_{l-2}}
\end{split} \]
holds for all $l\in\N$ and $x\in K_l\without K_{l-1}$ and $u\in C^\infty(M,\R)$ which satisfy $u>u_i$ on $K_{l+1}\without K_{l-2}$. We consider $k\define\max\set{k_1,\dots,k_m}$, $d\define\max\set{d_1,\dots,d_m}$, $\alpha\define\min\set{\alpha_1,\dots,\alpha_m}$ and the pointwise maxima $u_0\define\max\set{u_1,\dots,u_m}$, $b\define\max\set{b_1,\dots,b_m}$ in $C^0(M,\R_{\geq0})$. For every $i\in\set{1,\dots,m}$,
\[ \begin{split}
\Phi_i(u)(x) &\;\leq\; \sup\Set{ \e^{-\alpha u(y)}\,b(y)\cdot\left(1+\sum_{j=0}^k\abs{\nabla_\eta^ju}_\eta(y)\right)^d }{y\in K_{l+1}\without K_{l-2}}
\end{split} \]
holds for all $l\in\N$ and $x\in K_l\without K_{l-1}$ and $u\in C^\infty(M,\R)$ which satisfy $u>u_0$ on $K_{l+1}\without K_{l-2}$. This implies for all $l\in\N$ and $x\in K_l\without K_{l-1}$ and $u>u_0$:
\[ \begin{split}
\Phi(u)(x) &\;=\; Q\big(x,\Phi_1(u)(x),\dots,\Phi_m(u)(x)\big)
\;\leq\; c(x)\cdot\big(\Phi_1(u)(x)+\dots+\Phi_m(u)(x)\big)^r\\
&\;\leq\; m^r c(x)\left(\sup\Set{ b(y)\; \e^{-\alpha u(y)} \left(1+\sum_{j=0}^k\abs{\nabla_\eta^ju}_\eta(y)\right)^d }{y\in K_{l+1}\without K_{l-2}}\right)^r\\
&\;\leq\; \sup\Set{ \e^{-\alpha ru(y)}\; m^rc(y)b(y)^r\;\left(1+\sum_{j=0}^k\abs{\nabla_\eta^ju}_\eta(y)\right)^{dr} }{y\in K_{l+1}\without K_{l-2}} \,.
\end{split} \]
Thus $\Phi$ is a quasi-flatzoomer for $\Kex$.
\end{proof}

\section{Lower bounds on Riemannian injectivity and convexity radii} \label{radii}

While the injectivity radius function $\inj_g\colon M\to\oci{0}{\infty}$ of a not necessarily complete Riemannian manifold $(M,g)$ is defined for instance in \cite[p.~118]{Chavel}, an analogously general discussion of the convexity radius $\conv_g\colon M\to\oci{0}{\infty}$ is hard to find in the literature. It does not really matter how we generalize the definition from the complete case \cite[pp.~403, 406]{Chavel}, though, because the metrics we construct in the end in Theorem \ref{maingeneral} will be complete anyway; only the proofs involve metrics that are not a priori complete. So, for simplicity, let us define for a not necessarily complete $(M,g)$ the notion of a strongly convex subset in the usual way \cite[p.~403]{Chavel}, and let us define for $x\in M$ the convexity radius of $g$ at $x$ by
\[
\conv_g(x) \define \sup\Set{\rho\in\bigcci{0}{\inj_g(x)}}{\forall r\in\smallcoi{0}{\rho}\colon\,\mt{$\bigSet{z\in M}{\dist_g(x,z)<r}$ is $g$-strongly convex}} \in \oci{0}{\infty} \,.
\]
It is a priori $\leq$ any other definition we have seen and thus yields a priori the strongest Theorem \ref{maingeneral}.

\smallskip
Our aim in the present section is to prove that the inverse convexity radius (and thus also the inverse injectivity radius) of Riemannian metrics on a manifold $M$ --- or, more generally, on a foliation on $M$ --- is a quasi-flatzoomer $\Phi$ with respect to conformal factors. In Section \ref{mainproof}, we will construct a function $u\in C^\infty(M,\R_{\geq0})$ with $\Phi(u)\leq1$. Then we have $\inj_{g[u]}\geq1$, so in particular $g[u]$ is complete.

At the end of the section, we explain why this construction cannot be generalized to arbitrary semi-Riemannian metrics.

\medskip
The standard lower estimates of the injectivity radius of a Riemannian metric due to Heintze--Karcher \cite[Corollary 2.3.2]{HK} and Cheeger--Gromov--Taylor \cite[Theorem 4.7]{CGT} do apparently not imply the desired quasi-flatzoomer property directly in our situation. But we can just as well argue in a more elementary way. We use the following version of Klingenberg's lemma; it does not assume completeness but involves an assumption on all self-intersecting geodesics, not just on periodic geodesics. (Here a \textbf{self-intersecting geodesic} is the image of a geodesic $\gamma\colon\cci{a}{b}\to M$ with $\gamma(a)=\gamma(b)$.)
\begin{Klingenberg} \label{Klingenberg}
Let $(M,g)$ be a Riemannian manifold, let $x\in M$, let $\delta,\ell,r\in\R_{>0}$. Assume
\begin{itemize}
\item the ball $B^g_r(x)\define\Set{z\in M}{\dist_g(x,z)\leq r}$ (which is closed in $M$) is compact;
\item $\sec_g(\sigma)\leq\delta$ holds for every $y\in B^g_r(x)$ and every $2$-plane $\sigma\subseteq T_yM$;
\item every self-intersecting geodesic in $(M,g)$ which is contained in $B^g_r(x)$ has length $\geq\ell$.
\end{itemize}
Then $\inj_g(x) \geq \min\set{\frac{\pi}{\sqrt{\delta}},\, \frac{\ell}{2},\, r}$.
\end{Klingenberg}
As we do not know a reference for \ref{Klingenberg}, let us review the proof. We need the following two results.
\begin{lemma}[{\protect Morse--Sch\"onberg \cite[Theorem II.6.3]{Chavel}}] \label{MS}
Let $\delta,r\in\R_{>0}$, let $\gamma\colon\cci{0}{r}\to M$ be a unit-speed geodesic in a Riemannian manifold $(M,g)$ such that $\sec_g(\sigma)\leq\delta$ holds for all $t\in\cci{0}{r}$ and all $2$-planes $\sigma\subseteq T_{\gamma(t)}M$. If $r<\frac{\pi}{\sqrt{\delta}}$, then there is no conjugate point of $\gamma(0)$ along $\gamma$.\qed
\end{lemma}

Recall that the \textbf{conjugate radius} $\conj_g(x)\in\oci{0}{\infty}$ of a point $x$ in a (possibly incomplete) Riemannian manifold $(M,g)$ is the number $\inf\Set{\varrho(v)}{v\in T_xM,\, g(v,v)=1}$, where $\varrho(v)$ is the supremum of all $a\in\R_{>0}$ such that the maximal $g$-geodesic $\gamma$ with $\gamma'(0)=v$ is defined on $\cci{0}{a}$ and has no conjugate point of $\gamma(0)=x$ along $\bigrestrict{\gamma}{\cci{0}{a}}$.

\begin{lemma} \label{babyKling}
Let $(M,g)$ be a Riemannian manifold, let $x\in M$, let $\ell,r\in\R_{>0}$. Assume that
\begin{itemize}
\item $B\define B^g_r(x)$ is compact;
\item every self-intersecting geodesic in $(M,g)$ which is contained in $B$ has length $\geq\ell$.
\end{itemize}
Then $\inj_g(x) \geq \min\bigset{\conj_g(x),\, \ell/2,\, r}$.
\end{lemma}
\begin{proof}[Sketch of proof]
Assume $\rho\define\inj_g(x)<r$. Because of the compactness of $B$, the tangent space ball $B^g_{x,r}(0)\define \Set{v\in T_xM}{\abs{v}_g\leq r}$ is contained in the domain of $\exp^g_x$. Since $\inj_g(x)$ is the supremum of all $s\in\R_{>0}$ for which $\exp^g_x$ restricted to $B^g_{x,s}(0)$ is a smooth embedding, we have $\inj_g(x)=\conj_g(x)$ (if $\exp^g_x$ has a critical point at the boundary of $B^g_{x,\rho}(0)$) or $\exp^g_x$ is not injective on $B^g_{x,\rho}(0)$. In the latter case, there is a point $y\in\inte(B)$ with $\dist_g(x,y) = \inj_g(x)$ for which $x$ and $y$ are connected by two distinct length-minimizing geodesics such that $y$ is not conjugate to $x$ along either. Now repeat the proof of the complete case \cite[Lemma 5.6]{CE} verbatim to get a self-intersecting geodesic of length $2\inj_g(x)$.
\end{proof}

\begin{proof}[Proof of \ref{Klingenberg}]
The compactness assumption implies that every maximal unit-speed geodesic starting in $x$ is defined at least on $\cci{0}{r}$. By the sectional curvature assumption, Lemma \ref{MS} yields therefore $\conj_g(x)\geq\min\set{\pi/\delta^{1/2},\, r}$. From Lemma \ref{babyKling}, we obtain $\inj_g(x) \geq \min\set{\pi/\delta^{1/2},\, \ell/2,\, r}$.
\end{proof}

In order to derive a similar estimate for the convexity radius, we recall J.~H.~C.\ Whitehead's lower bound \cite[Theorem 5.14]{CE}:
\begin{lemma} \label{Whitehead}
Let $(M,g)$ be a Riemannian manifold, let $x\in M$, let $\delta,\iota,r\in\R_{>0}$. Assume that
\begin{itemize}
\item the ball $B\define B_r^g(x)$ is compact;
\item $\sec_g(\sigma)\leq\delta$ holds for every $y\in B$ and every $2$-plane $\sigma\subseteq T_yM$;
\item $\inj_g(y)\geq\iota$ holds for every $y\in B$.
\end{itemize}
Then $\conv_g(x) \geq \min\set{\frac{\pi}{2\sqrt{\delta}},\, \frac{\iota}{2},\, r}$.
\end{lemma}

\begin{corollary}[to \ref{Klingenberg} and \ref{Whitehead}] \label{convestimate}
Let $(M,g)$ be a Riemannian manifold, let $x\in M$, let $\delta,\ell,r\in\R_{>0}$. Assume
\begin{itemize}
\item the ball $B\define B_r^g(x)$ is compact;
\item $\sec_g(\sigma)\leq\delta$ holds for every $y\in B$ and every $2$-plane $\sigma\subseteq T_yM$;
\item every self-intersecting geodesic in $(M,g)$ which is contained in $B$ has length $\geq\ell$.
\end{itemize}
Then $\conv_g(x) \geq \frac{1}{2}\min\set{\frac{\pi}{\sqrt{\delta}},\, \frac{\ell}{2},\, \tfrac{r}{2}}$.
\end{corollary}
\begin{proof}
$B'\define B^g_{r/2}(x)\subseteq B$ is compact. For every $y\in B'$, the ball $B_y\define B^g_{r/2}(y)$ is contained in $B$. Thus $B_y$ is compact, $\sec_g\leq\delta$ holds on $B_y$, and every self-intersecting geodesic in $B_y$ has length $\geq\ell$. Hence Klingen\-berg's lemma implies $\inj_g(y) \geq \iota\define \min\bigset{\pi/\sqrt{\delta},\, \ell/2,\, r/2}$. Lemma \ref{Whitehead}, with $r/2$ and $B'$ in the roles of $r$ and $B$, yields
\[
\conv_g(x) \geq \tfrac{1}{2}\min\set{\tfrac{\pi}{\sqrt{\delta}},\, \iota,\, r}
= \tfrac{1}{2}\min\set{\tfrac{\pi}{\sqrt{\delta}},\, \tfrac{\ell}{2},\, \tfrac{r}{2}} \,.\qedhere
\]
\end{proof}

\medskip
Now we introduce the quantities that Theorem \ref{RiemannInj} is about:
\begin{definition}
Let $\Fol$ be a foliation on a manifold $M$, let $g$ be a Riemannian metric on $\Fol$ (cf.\ \ref{ex3}). For each leaf $L$ of $\Fol$, $g_L$ denotes the Riemannian metric on $L$ that is the restriction of $g$. We define $\conv_g^\Fol\colon M\to\oci{0}{\infty}$ to be the function whose restriction to each $\Fol$-leaf $L$ is $\conv_{g_L}\colon L\to\oci{0}{\infty}$. Analogously, $\inj_g^\Fol\colon M\to\oci{0}{\infty}$ denotes the function whose restriction to each $\Fol$-leaf $L$ is $\inj_{g_L}\in C^0(L,\oci{0}{\infty})$. Then $1/\conv_g^\Fol$ and $1/\inj_g^\Fol$ are functions $M\to\coi{0}{\infty}$.
\end{definition}

\begin{remark}
In the situation of the preceding definition, the functions $\conv_g^\Fol$ and $\inj_g^\Fol$ are in general not continuous. For example, take the foliation $\Fol$ on $M\define (\R^n\times\R)\without\set{(0_n,0)}$ whose leaves are the sets $L_0\define(\R^n\without\set{0_n})\times\set{0}$ and $L_t\define\R^n\times\set{t}$ with $t\in\R\without\set{0}$, and take $g$ to be the metric on $\Fol$ whose restriction to each $L_i$ is the euclidean metric there. At each point of $L_0$, $\conv_g^\Fol = \inj_g^\Fol$ is not continuous, because it is constant $\infty$ on $\bigcup_{t\in\R\without\set{0}}L_t$ but finite-valued on $L_0$.

This is the reason why in Definition \ref{quasidef} we allowed the $\Phi(u)$ to be arbitrary functions $M\to\R$ instead of continuous ones.
\end{remark}

Now we are ready to prove the main result of this section.

\begin{theorem} \label{RiemannInj}
Let $\Fol$ be a foliation on a manifold $M$, let $g$ be a Riemannian metric on $\Fol$, let $\Kex=(K_i)_{i\in\N}$ be a compact exhaustion of $M$. Then $\Phi\colon C^\infty(M,\R)\to \Fct(M,\R_{\geq0})$ given by
\[
\Phi(u) \define 1/\conv_{g[u]}^\Fol
\]
is a quasi-flatzoomer for $\Kex$. The same holds with $\inj_{g[u]}^\Fol$ instead of $\conv_{g[u]}^\Fol$.
\end{theorem}

\begin{proof}
Let $\mathcal{A}$ be a foliation atlas for $\Fol$. We choose a (parametrized) locally finite cover $\mathcal{U}=(U_i)_{i\in\N}$ of $M$ by open sets $U_i$ each of which has compact closure contained in the domain of some $\mathcal{A}$-chart $\varphi_i$.

Let $n\define\dim\Fol$. For $i\in\N$, $\varphi_i$ induces for each leaf $L$ coordinates on $U_i\cap L$. For any $u\in C^\infty(M,\R)$, we can consider the Christoffel symbols ${}^{g[u]}\Gamma_{ab}^c$ of the (leafwise) metric $g[u]$ with respect to these coordinates. Since $U_i$ has compact closure in $\dom(\varphi_i)$, there exists a constant $A_i\in\R_{>0}$ such that
\[
\abs{{}^{g[u]}\Gamma_{ab}^c} \leq A_i\left(1 +\bigabs{\!\diff u}_g\right)
\]
holds pointwise on $U_i\cap L$ for every $\Fol$-leaf $L$ and every $u\in C^\infty(M,\R)$: we have
\[ \begin{split}
{}^{g[u]}\Gamma_{ab}^c &= \frac{1}{2}\sum_{m=1}^ng[u]^{cm}\big(\partial_ag[u]_{bm} +\partial_bg[u]_{am} -\partial_mg[u]_{ab}\big)\\
&= \frac{1}{2}\sum_{m=1}^n\frac{g^{cm}}{\e^{2u}}\bigg(\e^{2u}\big(\partial_ag_{bm} +\partial_bg_{am} -\partial_mg_{ab}\big) +2\e^{2u}\big(\partial_au\;g_{bm} +\partial_bu\;g_{am} -\partial_mu\;g_{ab}\big)\bigg) \,.
\end{split} \]

For $i\in\N$, we denote the (leafwise) euclidean metric on $\bigrestrict{\Fol}{\dom(\varphi_i)}$, obtained via $\varphi_i$-pullback, by $\eucl_i$. There exists a constant $C_i\in\R_{>0}$ such that
\begin{align*}
C_i\abs{v}_{\eucl_i} &\geq \abs{v}_g \geq C_i^{-1}\abs{v}_{\eucl_i}
\end{align*}
holds for every $\Fol$-leaf $L$ and every $x\in U_i\cap L$ and every $v\in T_xL$. We define $H_i\define 4n^2A_iC_i^3\in\R_{>0}$.

Since $\mathcal{U}$ is locally finite, there exists an $H\in C^0(M,\R_{>0})$ with $\forall x\in M\colon \forall i\in\N\colon \big(x\in U_i \implies H(x)\geq H_i\big)$.

\smallskip
The Examples \ref{ex3} and \ref{ex5} (with $Q(x,s)=\frac{2}{\pi}s^{1/2}$) tell us that $\Phi_0\colon C^\infty(M,\R)\to C^0(M,\R_{\geq0})$ given by
\[
\Phi_0(u) \define \frac{2}{\pi}\abs{\Riem_{g[u]}}^{1/2}_{g[u]}
\]
is a flatzoomer. Moreover, $\Phi_1\colon C^\infty(M,\R)\to C^0(M,\R_{\geq0})$ given by
\[
\Phi_1(u) \define \e^{-u}H\cdot\left(1 +\bigabs{\!\diff u}_g\right)
\]
is obviously a flatzoomer.

Let $K_{-2}\define K_{-1}\define\leer$. There exists a (sufficiently large) function $u_1\in C^0(M,\R)$ such that for every $i\in\N$, for every leaf $L$ and for every $x\in\left(K_i\without K_{i-1}\right)\cap L$, there is a $j\in\N$ with
\[
B^{g[u_1]_L}_1(x) \subseteq U_j\cap\left(K_{i+1}\without K_{i-2}\right) \,.
\]
Trivially, also $\Phi_2\colon C^\infty(M,\R)\to C^0(M,\R_{\geq0})$ given by $\Phi_2(u) \define 4\e^{-u}\e^{u_1}$ is a flatzoomer.

\smallskip
By Example \ref{ex5}, $\Psi\define\Phi_0+\Phi_1+\Phi_2$ is a flatzoomer; i.e., there exist $k,d\in\N$, $\alpha\in\R_{>0}$, $u_0\in C^0(M,\R)$, $P\in C^0(M,\Poly{k+1}{d})$ and a Riemannian metric $\eta$ on $M$ such that
\[
\Psi(u)(x) \;\leq\; \e^{-\alpha u(x)}P(x)\left(u(x),\, \bigabs{\nabla_\eta^1u}_\eta(x),\, \dots,\, \bigabs{\nabla_\eta^ku}_\eta(x)\right)
\]
holds for all $x\in M$ and all $u\in C^\infty(M,\R)$ with $u(x)>u_0(x)$. Without loss of generality, we may assume that $u_0$ is $\geq$ than each of the analogous functions which appear in the flatzoomer conditions of $\Phi_0,\Phi_1,\Phi_2$.

\smallskip
We claim that
\[
1/\inj_{g[u]}^\Fol(x) \;\leq\;
1/\conv_{g[u]}^\Fol(x) \;\leq\; \sup\Set{\e^{-\alpha u(y)} P(y)\left(u(y),\, \bigabs{\nabla_\eta^1u}_\eta(y),\, \dots,\, \bigabs{\nabla_\eta^ku}_\eta(y)\right)}{y\in K_{i+1}\without K_{i-2}}
\]
holds for all $i\in\N$ and $x\in K_i\without K_{i-1}$ and $u\in C^\infty(M,\R)$ which satisfy $u>u_0$ on $K_{i+1}\without K_{i-2}$. This claim implies by Definition \ref{quasidef} that the theorem is true.

\smallskip
In order to prove the claim, only the second \scare{$\leq$} has to be checked. By Corollary \ref{convestimate}, it suffices to verify that for all $i\in\N$ and leaves $L$ and $x\in (K_i\without K_{i-1})\cap L$ and $u\in C^\infty(M,\R)$ which satisfy $u>u_0$ on $K_{i+1}\without K_{i-2}$, there exists an $r\in\R_{>0}$ such that $B^{g[u]_L}_r(x)$ is compact (in the leaf topology on $L$) and the following inequalities hold (where $\sup\leer\define 0$):
\begin{align}
\frac{2}{\pi}\abs{\max\BigSet{\sec_{g[u]_L}(\sigma)}{z\in B^{g[u]_L}_r(x),\; \sigma\in\Gr_2(T_zL)}}^{1/2}
&\leq \sup\BigSet{\Phi_0(u)(y)}{y\in K_{i+1}\without K_{i-2}} \,, \label{check1} \\
\sup\BigSet{4/\length(\gamma)}{\text{$\gamma\subset B^{g[u]_L}_r(x)$ is a self-inters.\ $g[u]_L$-geodesic}} &\leq \sup\BigSet{\Phi_1(u)(y)}{y\in K_{i+1}\without K_{i-2}} \,, \label{check2} \\
\frac{4}{r} &\leq \sup\BigSet{\Phi_2(u)(y)}{y\in K_{i+1}\without K_{i-2}} \,. \label{check3}
\end{align}

We will show that $r\define 1/\sup\Set{\e^{u_1(y)-u(y)}}{y\in K_{i+1}\without K_{i-2}}$ has these properties. It satisfies \eqref{check3} tautologically. Moreover, with $q\define \inf\Set{\e^{u(y)-u_1(y)}}{y\in K_{i+1}\without K_{i-2}}$ we obtain
\[
B\define B^{g[u]_L}_r(x) = B^{\exp(2u-2u_1)g[u_1]_L}_r(x) \subseteq B^{q^2g[u_1]_L}_r(x)
= B^{g[u_1]_L}_{r/q}(x) = B^{g[u_1]_L}_1(x) \subseteq U_j\cap\left(K_{i+1}\without K_{i-2}\right)
\]
for some $j\in\N$. The ball $B$ is a connected closed subset of $L$ with respect to the leaf topology, and $B$ is contained in $U_j$, whose closure in $M$ is a compact subset of a foliation chart domain. All this together implies that $B$ is compact in the leaf topology on $L$ (and also in the topology on $M$).

\smallskip
Inequality \eqref{check1} is true: For each $z\in B$ and each $\sigma\in\Gr_2(T_zL)$, we choose a $g[u]_L$-orthonormal basis $(e_1,e_2)$ of $\sigma$. This yields $\abs{\sec_{g[u]_L}(\sigma)} = \abs{\Riem_{g[u]_L}(e_1,e_2,e_1,e_2)} \leq \abs{\Riem_{g[u]_L}}_{g[u]_L}$. Since $z$ lies in $K_{i+1}\without K_{i-2}$, the definition of $\Phi_0(u)$ implies \eqref{check1}.

\vspace*{-1.0ex}
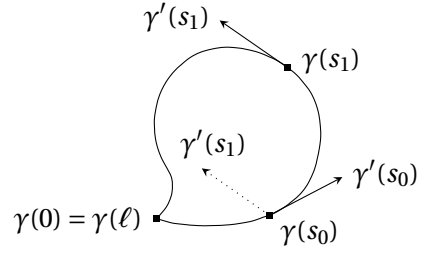
\begin{figure}[h]
\begin{minipage}{0.59\textwidth} \setlength\parindent{1em}
It remains to check \eqref{check2}. Let $\gamma\colon\cci{0}{\ell}\to B$ be an arclength-pa\-ram\-e\-trized $g[u]_L$-geodesic with $\gamma(0) = \gamma(\ell)$. There exists an $s_0\in\cci{0}{\ell}$ with $u(\gamma(s_0)) = \min_{s\in\cci{0}{\ell}}u(\gamma(s))$.

\medskip
Since $B\subseteq U_j\subseteq\dom(\varphi_j)$, the euclidean metric $\eucl_j$ is defined on $B$ and we can regard $B$ as a subset of the vector space $\R^n$. There is an $s_1\in\cci{0}{\ell}$ with $\eval{\gamma'(s_1)}{\gamma'(s_0)}_{\eucl_j} \leq 0$, because the map $w\colon\cci{0}{\ell}\ni t\mapsto \eval{\gamma'(t)}{\gamma'(s_0)}_{\eucl_j}$ satisfies $\int_0^\ell w(t)\diff t = \eval{\gamma(\ell)-\gamma(0)}{\gamma'(s_0)}_{\eucl_j} = 0$.

\medskip
In particular, we have $\abs{\gamma'(s_0)}_{\eucl_j} \leq \abs{\gamma'(s_1)-\gamma'(s_0)}_{\eucl_j}$.
\end{minipage}
\begin{minipage}{0.39\textwidth}
\centering
\begin{tikzpicture}[>=stealth,scale=0.16]
\draw plot[smooth,tension=0.7] coordinates { (9.7362,2.1197) (10.861,1.7682) (11.8803,1.6276) (13.1105,1.5221) (14.2704,1.5221) (15.7467,1.5924) (17.1527,1.8033) (17.8557,1.9791) (18.5,2.22) (19,2.38) (19.5,2.6) (20,2.9) (20.5,3.24) (20.75,3.45) (21,3.65) (21.5,4.1) (22,4.73) (22.3755,5.407) (22.727,6.2506) (22.9379,7.0942) (23.1136,7.973) (23.2191,8.9571) (23.1615,10.0116) (23.0226,10.9136) (22.7514,11.7924) (22.3647,12.6008) (21.8726,13.2686) (21.2399,14.0067) (20.4667,14.6394) (19.6934,15.1666) (18.9552,15.5181) (18.1116,15.8696) (17.1626,16.1157) (16.2839,16.2563) (15.2997,16.2988) (14.2804,16.1508) (13.2962,15.8696) (12.4877,15.4478) (11.7848,14.8854) (11.1872,14.3231) (10.6486,13.6406) (10.2356,12.9376) (9.9363,12.2346) (9.7254,11.4613) (9.5554,10.442) (9.5572,9.5984) (9.638,8.7196) (9.8542,7.8409) (10.1931,6.986)  (10.5798,6.283) (10.9724,5.4156) (11.0015,4.3498) (10.7555,3.5062) (10.3689,2.8384) (9.7362,2.1197) };
\node[fill=black,inner sep=0.25ex] at (9.7362,2.1197) {};
\node[outer sep=0ex,inner sep=0.3ex,label=0:${\gamma(0)=\gamma(\ell)}$] at (-3.2,2.0) {};
\node[fill=black,inner sep=0.25ex] at (19,2.38) {};
\node[outer sep=0ex,inner sep=0.3ex,label=275:$\gamma(s_0)$] at (19.0,3.1) {};
\node[fill=black,inner sep=0.25ex] at (20.4667,14.6394) {};
\node[outer sep=0ex,inner sep=0.3ex,label=0:$\gamma(s_1)$] at (20.5,15.2) {};
\draw[->] (19,2.38) -- (25,5.54);
\node[outer sep=0ex,inner sep=0.3ex,label=0:$\gamma'(s_0)$] at (24.7,5.8) {};
\draw[->] (20.4667,14.6394) -- (14.9083,18.5239);
\node[outer sep=0ex,inner sep=0.3ex,label=0:$\gamma'(s_1)$] at (7.5,18.8) {};
\draw[->,dotted] (19,2.38) -- (13.4416,6.2645);
\node[inner sep=0.2ex,label=85:$\gamma'(s_1)$] at (10.8,6.0) {};
\end{tikzpicture}
\caption{A self-inter\-secting $g[u]_L$-geodesic $\gamma$ in $B\subseteq\R^n$.}
\end{minipage}
\end{figure}

\vspace*{-1.0ex}
Denoting the components (with respect to the chosen coordinates) of a vector $v\in T_xL$ with $x\in B$ by $v_1,\dots,v_n$, we have the following estimates:
\begin{align*}
C_j\abs{v}_{\eucl_j} &\geq \abs{v}_g \geq C_j^{-1}\abs{v}_{\eucl_j} \,,
&n^{1/2}\abs{v}_{\eucl_j} &\geq \sum_{a=1}^n\abs{v_a} \,.
\end{align*}
In particular,
\[
\forall s\in\cci{0}{\ell}\colon\; n^{1/2}C_j\e^{-u(\gamma(s))}
= n^{1/2}C_j\e^{-u(\gamma(s))}\abs{\gamma'(s)}_{g[u]}
= n^{1/2}C_j\abs{\gamma'(s)}_g
\geq \sum_{a=1}^n\abs{\gamma_a'(s)} \,.
\]

Using this and $\forall c\colon \abs{\partial_c}_{\eucl_j}=1$ and the $g[u]_L$-geodesic equation
\[
\forall s\in\cci{0}{\ell}\colon\; \gamma''(s) = \sum_{c=1}^n\gamma_c''(s)\;\partial_c\big(\gamma(s)\big) = -\sum_{a,b,c=1}^n{}^{g[u]_L}\Gamma_{ab}^c\big(\gamma(s)\big)\; \gamma_a'(s)\gamma_b'(s)\;\partial_c\big(\gamma(s)\big) \,,
\]
we obtain
\[ \begin{split}
1 &\;=\; \abs{\gamma'(s_0)}_{g[u]}
\;=\; \e^{u(\gamma(s_0))}\abs{\gamma'(s_0)}_g\\
&\;\leq\; C_j\e^{u(\gamma(s_0))}\abs{\gamma'(s_0)}_{\eucl_j}\\
&\;\leq\; C_j\e^{u(\gamma(s_0))}\abs{\gamma'(s_1)-\gamma'(s_0)}_{\eucl_j}
\;=\; C_j\e^{u(\gamma(s_0))}\abs{\int_{s_0}^{s_1}\gamma''(s)\diff s}_{\eucl_j}\\
&\;\leq\; C_j\e^{u(\gamma(s_0))}\abs{ \sum_{a,b,c=1}^n\; \int_{s_0}^{s_1} \Bigabs{{}^{g[u]_L}\Gamma_{ab}^c\big(\gamma(s)\big)} \cdot \bigabs{\gamma_a'(s)} \cdot \bigabs{\gamma_b'(s)} \;\diff s }\\
&\;\leq\; C_j\e^{u(\gamma(s_0))} \abs{ \sum_{a,b,c=1}^n\; \int_{s_0}^{s_1} A_j\cdot\Big(1+\bigabs{\!\diff u}_g(\gamma(s))\Big) \cdot \bigabs{\gamma_a'(s)} \cdot \bigabs{\gamma_b'(s)} \;\diff s }\\
&\;=\; nA_jC_j\e^{u(\gamma(s_0))}\; \abs{ \int_{s_0}^{s_1} \Big(1+\bigabs{\!\diff u}_g(\gamma(s))\Big)\cdot \left(\sum_{a=1}^n\abs{\gamma_a'(s)}\right)^2\diff s }\\
&\;\leq\; n^2A_jC_j^3\; \abs{ \int_{s_0}^{s_1} \Big(1+\bigabs{\!\diff u}_g(\gamma(s))\Big) \cdot \e^{u(\gamma(s_0))}\;\e^{-2u(\gamma(s))}\diff s }\\
&\;\leq\; n^2A_jC_j^3\; \abs{ \int_{s_0}^{s_1} \Big(1+\bigabs{\!\diff u}_g(\gamma(s))\Big)\cdot\e^{-u(\gamma(s))}\diff s }\\
&\;\leq\; \ell n^2A_jC_j^3\;\norm{\e^{-u}\Big(1+\bigabs{\!\diff u}_g\Big)}_{C^0(U_j\cap(K_{i+1}\without K_{i-2}))} \,,
\end{split} \]
and thus
\[ \begin{split}
4/\ell
&\leq H_j\norm{\e^{-u}\Big(1+\bigabs{\!\diff u}_g\Big)}_{C^0(U_j\cap(K_{i+1}\without K_{i-2}))}\\
&\leq \norm{H\e^{-u}\Big(1+\bigabs{\!\diff u}_g\Big)}_{C^0(K_{i+1}\without K_{i-2})}
= \sup\BigSet{\Phi_1(u)(y)}{y\in K_{i+1}\without K_{i-2}} \,.
\end{split} \]
Hence also \eqref{check2} is true. This completes the proof.
\end{proof}

It remains to explain why the statements \ref{gen2} and \ref{gen4} of Theorem \ref{maingeneral} cannot be generalized to arbitrary semi-Riemannian metrics. One problem is that not every conformal class of, say, Lorentzian metrics contains a complete metric. (Recall that since there is no Lorentzian analogue of the Hopf--Rinow theorem, the notion of \emph{completeness} of Lorentzian metrics refers always to geodesic completeness.)

\begin{example} \label{incomplete}
Let $m\in\N$, let $M$ be a manifold which contains an open subset $U$ diffeomorphic to $\R\times\Sph^1\times\R^m$; we identify $U$ and $\R\times\Sph^1\times\R^m$ by the diffeomorphism. Let $g_0$ be a Lorentzian metric on $M$ which has in a neighborhood of the circle $L\define\set{0}\times\Sph^1\times\set{0_m}\subset M$ the form
\[
(g_0)_{(x,y,z)}= \smatrix{ 0&1\\ 1&x\\ &&1\\ &&&\ddots\\ &&&&1 } \,,
\]
where $x$ and $z$ are the standard coordinates on $\R$ resp.\ $\R^m$ and where $y\in\Sph^1$. Then the conformal class of $g_0$ contains no metric all of whose lightlike geodesics are complete: For every $g$ in the conformal class of $g_0$, the maximal domain $I\subseteq\R$ of the (lightlike) $g$-geodesic $\gamma\in C^\infty(I,M)$ with $\gamma(0)=(0,0,0)$ and $\gamma'(0)=(0,1,0)$ is bounded from above. (Here we consider $0\in\R/\Z=\Sph^1$.) The image of $\gamma$ is $L$.
\end{example}
\begin{proof}
Let $u\in C^\infty(M,\R_{>0})$, let $g=ug_0$. We compute $\gamma$ in the universal covering of $\R\times\Sph^1\times\R^m$, where we can use the standard global coordinates $(x_1,\dots,x_{m+2})$ (with $x=x_1$, $y=x_2$). The components $\gamma_1,\dots,\gamma_{m+2}$ solve the geodesic equation
\[
\forall k\in\set{1,\dots,m+2}\colon \forall t\in I\colon\;\;
\gamma_k''(t) = -\sum_{i,j=1}^{m+2}\Gamma_{ij}^k\left(\gamma(t)\right)\gamma_i'(t)\gamma_j'(t) \,,
\]
where
\[
\Gamma_{ij}^k = \frac{1}{2}\sum_{l=1}^{m+2}g^{kl}\left(\partial_ig_{jl} +\partial_jg_{il} -\partial_lg_{ij}\right) \in C^\infty(\R^{m+2},\R)
\]
are the Christoffel symbols of $g$.

For $k\neq2$, all $\Gamma_{22}^k$ vanish on $\tilde{L}\define\set{0}\times\R\times\set{0_m}\subset\R^{m+2}$: On $\tilde{L}$, we have $\left(g^{kl}\right) = \frac{1}{u}\cdot\left(\smatrix{0&1\\1&0}\oplus\diag(1,\dots,1)\right)$ and thus, for $\kappa\geq3$:
\begin{align*}
\Gamma_{22}^1 &= \frac{1}{2}\sum_lg^{1l}\left(2\partial_2g_{2l} -\partial_lg_{22}\right)
= \frac{1}{2u}\left(2\partial_2g_{22} -\partial_2g_{22}\right)
= \frac{1}{2u}\partial_2(xu) = 0 \,,\\
\Gamma_{22}^\kappa &= \frac{1}{2}\sum_lg^{\kappa l}\left(2\partial_2g_{2l} -\partial_lg_{22}\right)
= \frac{1}{u}\partial_2g_{2\kappa} -\frac{1}{2u}\partial_\kappa g_{22}
= -\frac{1}{2u}\partial_\kappa(xu) = 0 \,.
\end{align*}
Hence for all $y,r\in\R$, the $g$-geodesic equation has a local solution $\gamma_{y,r}$ with $\gamma_{y,r}(0)=(0,y,0_m)$ and $\gamma_{y,r}'(0)=(0,r,0_m)$ such that all components of $\gamma_{y,r}$ except the $2$-component vanish identically. This implies that the image of the maximal geodesic $\gamma$ with $\gamma(0)=(0,0,0_m)$ and $\gamma'(0)=(0,1,0_m)$ is $\tilde{L}$.

To determine $\gamma$, we thus have to calculate only $\gamma_2$. On $\tilde{L}$, we compute
\[
\Gamma_{22}^2 = \frac{1}{2}\sum_lg^{2l}\left(2\partial_2g_{2l} -\partial_lg_{22}\right)
= \frac{1}{2u}\left(2\partial_2g_{21} -\partial_1g_{22}\right)
= \frac{2\partial_2u -\partial_1(xu)}{2u}
= \frac{2\partial_2u -u}{2u}
= \frac{\partial_2u}{u} -\frac{1}{2} \,.
\]
With $w\in C^\infty(\R,\R)$ given by $w(y)\define\ln\left(u(0,y,0_m)\right)$, $\gamma_2\in C^\infty(I,\R)$ is the maximal solution of
\begin{align*}
\gamma_2''(t) &= \left(\frac{1}{2}-w'\left(\gamma_2(t)\right)\right)\gamma_2'(t)^2 \,,
&\gamma_2(0)&=0 \,, \quad\gamma_2'(0)=1 \,.
\end{align*}
Since $w$ is the pullback of a function on $\Sph^1$ via the universal covering $\R\ni y\mapsto[y]\in\R/\Z$, it is $1$-periodic. In particular, $w-w(0)$ is bounded from above by some $C\in\R$. We obtain
\[
\forall t\in I\colon\;\; \left(\ln\compose\gamma_2'\right)'(t)
= \frac{\gamma_2''(t)}{\gamma_2'(t)}
= \frac{1}{2}\gamma_2'(t) -w'\left(\gamma_2(t)\right)\gamma_2'(t)
= \left(\frac{1}{2}\gamma_2 -w\compose\gamma_2\right)'(t)
\]
and thus
\[
\forall t\in I\colon\;\; \ln\left(\gamma_2'(t)\right) = \int_0^t\left(\ln\compose\gamma_2'\right)'(s)\diff s
= \frac{1}{2}\gamma_2(t) -w\left(\gamma_2(t)\right) +w(0)
\geq \frac{1}{2}\gamma_2(t) -C \,.
\]
This implies $\forall t\in I\colon \gamma_2'(t) \geq \e^{-C}\e^{\gamma_2(t)/2}$, hence
\[
\forall t\in I\cap\R_{\geq0}\colon\;\; 1
> 1-\e^{-\gamma_2(t)/2}
= \frac{1}{2}\int_0^{\gamma_2(t)}\frac{1}{\e^{\xi/2}}\diff\xi
= \frac{1}{2}\int_0^t\frac{\gamma_2'(s)}{\e^{\gamma_2(s)/2}}\diff s
\geq \frac{t}{2\e^C} \,,
\]
i.e., $t<2\e^C$. This proves that the domain $I$ of $\gamma$ is bounded from above.
\end{proof}

\begin{remark}
Note that the manifold $M$ in Example \ref{incomplete} can even be compact, e.g.\ if $M$ is the $n$-torus $\T^n$ for some $n\geq2$. It is well-known that some compact manifolds admit incomplete Lorentzian metrics (cf.\ e.g.\ \cite{BEE}), and \ref{incomplete} is essentially the standard example for this phenomenon; but as far as we can tell, the literature does not mention that it even yields a conformal class without complete metric. Besides, it is apparently an open question whether every manifold which admits a Lorentzian metric admits a complete one. (We are grateful to Stefan Suhr for remarks on these points.) We will not discuss here to which extent the completeness problem can be avoided by imposing causality conditions on the conformal class in question.
\end{remark}

\begin{remark}
There are no natural useful notions of \emph{convexity radius} or \emph{injectivity radius} of a Lorentzian manifold $(M,g)$, but one can define such radii via an auxiliary Riemannian metric $\eta$ on $M$: the \scare{size} of subsets of the domain of $\exp^g_x$ in each tangent space $T_xM$ can be measured in terms of $\eta$. The resulting \scare{mixed} injectivity radius has been studied by Chen--LeFloch \cite{CLF} and Grant--LeFloch \cite{GLF} in the situation when $\eta$ has the form $\Wick(g,t)$ for some temporal function $t$, as in our Theorem \ref{mainlorentz}. Example \ref{incomplete} suggests that statement \ref{gen2} of Theorem \ref{maingeneral} becomes true for an arbitrary \emph{semi}-Riemannian metric $g_0$ and an arbitrary additional Riemannian metric $\eta$ if one drops the completeness claim and replaces the (undefined) radius $\conv_{g_0[u]}$ (resp.\ $\inj_{g_0[u]}$) by the \scare{mixed} radius $\conv_{g_0[u]}^{\eta[u]}$ (resp.\ $\inj_{g_0[u]}^{\eta[u]}$). Analogously one might perhaps get a correct semi-Riemannian generalization of statement \ref{gen4} in Theorem \ref{maingeneral}. We will not investigate these matters here.
\end{remark}

\section{Proof of the main results} \label{mainproof}

We will obtain Theorem \ref{maingeneral} as a corollary to the following result about sequences of quasi-flatzoomers:
\begin{theorem} \label{FlatzoomThemAll}
Let $\Kex=(K_i)_{i\in\N}$ be a smooth compact exhaustion of a manifold $M$, let $(\Phi_i)_{i\in\N}$ be a sequence of quasi-flatzoomers for $\Kex$, let $(\eps_i)_{i\in\N}$ be a sequence in $C^0(M,\R_{>0})$, let $w\in C^0(M,\R)$. Then there exists a real-analytic $u\colon M\to\R$ with $u>w$ such that
\begin{equation} \label{pudelskern}
\forall i\in\N\colon\; \Phi_i(u) < \eps_i  \;\;\text{holds on $M\without K_i$} \,.
\end{equation}
\end{theorem}

We need some preparations for the rather technical proof of \ref{FlatzoomThemAll}. In that proof, we have to construct a function $u$ which increases rapidly, because we want the exponential factor $\e^{-\alpha u}$ from the quasi-flatzoomer definition to decrease rapidly. Of course, such a rapid increase makes the derivatives of $u$ large as well, and that is potentially harmful. The details of how we increase $u$ are therefore crucial; the most obvious attempt to do this would not work, as we indicate in Remark \ref{alpineremark} below. Lemma \ref{alpinist} is the analytic key to our argument.

\begin{definition} \label{climbers}
As usual, $\phi^{(i)}$ denotes the $i$th derivative of a function $\phi\in C^\infty(I,\R)$ on some interval $I\subseteq\R$. For $r\in\R_{\geq0}$, we define
\[
\Climbers(r) \define \Set{\phi\in C^\infty\big(\cci{0}{1},\cci{0}{r}\big)}{\phi(0)=0,\;\; \phi(1)=r,\;\; \forall i\in\N_{\geq1}\colon \phi^{(i)}(0)=0=\phi^{(i)}(1)} \,.
\]
A sequence $(\phi_n)_{n\in\N}$ in $C^\infty(\cci{0}{1},\R)$ is an \textbf{alpinist} iff $\forall n\in\N\colon \phi_n\in \Climbers(n)$.

\smallskip
Let $k\in\N$, let $a\in\R_{>0}$, let $\Theta=(\phi_n)_{n\in\N}$ be a sequence in $C^\infty(\cci{0}{1},\R)$. We define the set
\[
G_{k,a}[\Theta] \define
\Set{ \max_{t\in\cci{0}{1}}\e^{-a\phi_n(t)}\left(1+\sum_{j=0}^k\bigabs{\phi_n^{(j)}(t)}\right) }{n\in\N}
\;\subset\; \R_{>0} \,.
\]
\end{definition}

\begin{lemma} \label{alpinist}
Let $a\in\R_{>0}$, let $k\in\N$. There is an alpinist $\Theta$ such that the set $G_{k,a}[\Theta]$ is bounded.
\end{lemma}
(Here ``bounded'' means bounded from above, not away from $0$.)

\begin{proof}[Proof of Lemma \ref{alpinist}]
We let $c\define a/k$ if $k\geq1$, and $c\define 19.26$ if $k=0$. For $n\in\N$, we consider
\[
q_n\define 1-\e^{-nc}\in\coi{0}{1} \,.
\]
We choose some $\xi\in\Climbers(1)$ and define a sequence $\Theta = (\phi_n)_{n\in\N}$ in $C^\infty([0,1],\R)$ by
\begin{equation} \label{alpidef}
\phi_n \define -\frac{1}{c}\;\ln\big(1-q_n\xi\big) \,.
\end{equation}
The $\phi_n$ are well-defined because $\xi$ is $\cci{0}{1}$-valued and hence $1-\big(1-\e^{-nc}\big)\xi$ is $\cci{\e^{-nc}}{1}$-valued. Since $\xi\in\Climbers(1)$, we have $\phi_n(0)=0$ and $\phi_n(1) = -\frac{1}{c}\ln\left(1-q_n\right) = n$.

\smallskip
We claim that for every $i\in\N_{\geq1}$, there is a polynomial $P_i\in\R[X_0,\dots,X_i]$ with
\begin{align*}
\forall X_0\in\R &\colon\;\; P_i(X_0,0,\dots,0) = 0 \,,\\
\forall n\in\N &\colon\;\; \phi_n^{(i)} = \frac{P_i\left(q_n\xi^{(0)},\dots,q_n\xi^{(i)}\right)}{\left(1-q_n\xi\right)^i} \,.
\end{align*}
We prove this by induction over $i$. For $i=1$, the first derivatives $\phi_n' = \frac{1}{c}\left.q_n\xi'\middle/\left(1-q_n\xi\right)\right.$ have the claimed form. If the $i$th derivatives $\phi_n^{(i)}$ have the claimed form, then the $(i+1)$st derivatives
\[ \begin{split}
\phi_n^{(i+1)} &= \left(\frac{P_i\left(q_n\xi^{(0)},\dots,q_n\xi^{(i)}\right) }{\left(1-q_n\xi\right)^i}\right)'\\
&= \frac{\left(1-q_n\xi\right)\cdot\sum_{\nu=0}^i\frac{\partial P_i}{\partial X_\nu} \left(q_n\xi^{(0)},\dots,q_n\xi^{(i)}\right)q_n\xi^{(\nu+1)}
+iq_n\xi'\cdot P_i\left(q_n\xi^{(0)},\dots,q_n\xi^{(i)}\right) }{\left(1-q_n\xi\right)^{i+1}}
\end{split} \]
have the claimed form as well. This completes the proof of the claim.

\smallskip
For each $n\in\N$, we obtain $\forall i\in\N_{\geq1}\colon \phi_n^{(i)}(0) = 0 = \phi_n^{(i)}(1)$ and thus $\phi_n\in\Climbers(n)$. Hence $\Theta$ is an alpinist.

\smallskip
Since $\forall n\in\N\colon \smallabs{q_n}\leq1$, there exists for each $i\in\N_{\geq1}$ a constant $C_i\in\R_{>0}$ with
\[
\forall n\in\N\colon \norm{P_i\left(q_n\xi^{(0)},\dots,q_n\xi^{(i)}\right)}_{C^0(\cci{0}{1},\R)} \leq C_i \,.
\]

The supremum $S\define\sup\bigSet{(1+s)/\e^{as}}{s\in\R_{\geq0}}$ exists in $\R_{>0}$. We obtain for all $n\in\N$ and $t\in[0,1]$:
\[ \begin{split}
\frac{1+\sum_{i=0}^k\bigabs{\phi_n^{(i)}(t)}}{\e^{a\phi_n(t)}}
&\leq \frac{1+\bigabs{\phi_n(t)}}{\e^{a\phi_n(t)}}
+\sum_{i=1}^k\frac{ \abs{P_i\left(q_n\xi^{(0)}(t),\dots,q_n\xi^{(i)}(t)\right)} }{ \left(1-q_n\xi(t)\right)^i\;\e^{a\phi_n(t)} }\\
&\leq S +\sum_{i=1}^k\frac{C_i}{\left(1-q_n\xi(t)\right)^i\;\e^{a\phi_n(t)}}
= S +\sum_{i=1}^k \frac{C_i\cdot\left(1-q_n\xi(t)\right)^k}{\left(1-q_n\xi(t)\right)^i}
\leq S +\sum_{i=1}^k C_i\cdot 1^{k-i} \,.
\end{split} \]
(Note that $\sum_{i=1}^0(...) = 0$.) Hence the set $G_{k,a}[\Theta]$ is bounded by $S+\sum_{i=1}^k C_i$.
\end{proof}

\begin{remark} \label{alpineremark}
If you suspect that the proof of \ref{alpinist} --- in particular definition \eqref{alpidef} --- is unnecessarily complicated, the following example might be instructive. Consider any $\phi\in\Climbers(1)$ which is for some $\delta\in\ooi{0}{1}$ equal to $t\mapsto\e^{-1/t}$ on $\oci{0}{\delta}$. The sequence $\Theta=(\phi_n)_{n\in\N}$ given by $\phi_n = n\phi$ is an alpinist, but $G_{1,1}[\Theta]$ is not bounded.
\end{remark}
\begin{proof}[Proof of the claim made in Remark \ref{alpineremark}]
There is an $n_0\in\N_{\geq1}$ with $1/\ln(n_0)\leq\delta$. For $n\in\N$ with $n\geq n_0$, we consider $t_n\define 1/\ln(n)\in\oci{0}{\delta}$. We have $\phi(t_n) = \e^{-1/t_n} = 1/n$ and thus $\e^{-\phi_n(t_n)} = 1/\e$. Moreover, $\phi'(t_n) = \left.\e^{-1/t_n}\middle/t_n^2\right. = \left.\ln(n)^2\middle/n\right.$. Hence $\e^{-\phi_n(t_n)}\left(1+\abs{\phi_n(t_n)}+\abs{\phi_n'(t_n)}\right) = \frac{1}{\e}\left(1+1+\ln(n)^2\right)$. Since this tends to $\infty$ as $n\to\infty$, the set $G_{1,1}[\Theta]$ is not bounded.
\end{proof}

\begin{proof}[Proof of Theorem \ref{FlatzoomThemAll}]
Let $K_{-2}\define K_{-1}\define\leer$. For every $i\in\N$, the boundary $\Sigma_i$ of the smooth codimension-zero submanifold-with-boundary $K_i$ has an interior collar neighborhood $A_i \subseteq K_i\without K_{i-1}$ which can be diffeomorphically identified with $\cci{0}{1}\times\Sigma_i$ such that $\Sigma_i$ is identified with $\set{1}\times\Sigma_i$. Let $\rho_i\colon A_i\to\cci{0}{1}$ denote the projection to the first factor.

\smallskip
We fix a Riemannian metric $\eta$ on $M$. For $i,k\in\N$, the chain and product rules yield a constant $L_{i,k}\in\R_{>0}$ such that for all $x\in A_i$ and $f\in C^\infty([0,1],\R)$, we have
\begin{equation} \label{chainrule}
1+\sum_{j=0}^k\abs{\nabla_\eta^j\left(f\compose\rho_i\right)}_\eta(x)
\leq L_{i,k}\cdot\left(1+\sum_{j=0}^k\abs{f^{(j)}\big(\rho_i(x)\big)}\right) \,.
\end{equation}

For each $i\in\N$, the quasi-flatzoomer property of $\Phi_i$ gives us $k_i,d_i\in\N_{\geq1}$ and $\theta_i,w_i\in C^0(M,\R_{>0})$ and $a_i\in\R_{>0}$ such that
\begin{equation} \label{Phi}
\Phi_i(u)(x) \;\leq\; \sup\Set{\e^{-a_iu(y)}\,\theta_i(y)\cdot\left(1+\sum_{j=0}^{k_i}\abs{\nabla_\eta^ju}_\eta(y)\right)^{d_i} }{y\in K_{l+1}\without K_{l-2}}
\end{equation}
holds for all $l\in\N$ and $x\in K_l\without K_{l-1}$ and $u\in C^\infty(M,\R)$ which satisfy $u>w_i$ on $K_{l+1}\without K_{l-2}$.

\smallskip
For each $i$, we replace $\Phi_i$ by $u\mapsto \Phi_i(u)^{1/d_i}$, replace $a_i$ by $a_i/d_i$, replace $\theta_i$ by $\theta_i^{1/d_i}$, and replace $\eps_i$ by $\eps_i^{1/d_i}$. After this, we may assume without loss of generality that \eqref{Phi} holds with $d_i=1$.

\smallskip
There is a function $\hat{w}\in C^0(M,\R_{>0})$ with $\hat{w}>w$ such that for each $l\in\N$, $\hat{w}>\max\set{w,w_0,\dots,w_{l-1}}$ holds pointwise on $K_{l+1}\without K_{l-2}$.

\smallskip
For $i,l\in\N$, we define $\check{\eps}_{i,l}\in\R_{>0}$ and $\hat{\Phi}_{i,l}\colon C^\infty(M,\R)\to\R_{\geq0}$ by
\begin{align*}
\check{\eps}_{i,l} &\define \inf\,\BigSet{\,\eps_i(x)}{x\in K_l\without K_{l-1}} \,,\\
\hat{\Phi}_{i,l}(u) &\define \sup\Set{\e^{-a_iu(y)}\,\theta_i(y)\cdot\left(1+\sum_{j=0}^{k_i}\abs{\nabla_\eta^ju}_\eta(y)\right)}{y\in K_{l+1}\without K_{l-2}} \,.
\end{align*}
Inequality \eqref{Phi} implies that, for any $u\in C^\infty(M,\R)$ which satisfies $u\geq\hat{w}$ (and hence satisfies for every $l\in\N$ the inequality $\bigrestrict{u}{K_{l+1}\without K_{l-2}} > \bigrestrict{\max\set{w_0,\dots,w_{l-1}}}{K_{l+1}\without K_{l-2}}$), the statement \eqref{pudelskern} is true if
\begin{equation} \label{proto}
\forall i,l\in\N\colon \Big(l\geq i+1 \;\implies\; \hat{\Phi}_{i,l}(u) < \check{\eps}_{i,l} \Big) \,.
\end{equation}
(If \eqref{proto} holds, then for all $i,l\in\N$ with $l\geq i+1$, we have on $K_l\without K_{l-1}$: $\Phi_i(u) \leq \hat{\Phi}_{i,l}(u) < \check{\eps}_{i,l} \leq \eps_i$, because $u>w_i$ is fulfilled on $K_{l+1}\without K_{l-2}$. Thus for all $i\in\N$, $\Phi_i(u)<\eps_i$ holds on $M\without K_i$; i.e., \eqref{pudelskern} is true.)

\smallskip
For $l\in\N$, we define $\tilde{\eps}_l\in\R_{>0}$ and $\tilde{\Phi}_l\colon C^\infty(M,\R)\to\R_{\geq0}$ by
\begin{align*}
\tilde{\eps}_l &\define \min\BigSet{\check{\eps}_{i,l}}{i\in\set{0,\dots,l-1}} \,,\\
\tilde{\Phi}_l(u) &\define \max\BigSet{\hat{\Phi}_{i,l}(u)}{i\in\set{0,\dots,l-1}} \,.
\end{align*}
For any $u\in C^\infty(M,\R)$ which satisfies $u\geq\hat{w}$, the statement \eqref{pudelskern} is true if
\begin{equation} \label{aim}
\forall l\in\N\colon \tilde{\Phi}_l(u) < \tilde{\eps}_l \,.
\end{equation}
(This follows from \eqref{proto}, because \eqref{aim} implies for all $i,l\in\N$ with $l\geq i+1$: $\hat{\Phi}_{i,l}(u) \leq \tilde{\Phi}_l(u) < \tilde{\eps}_l \leq \check{\eps}_{i,l}$.)

\smallskip
For $l\in\N$, we define $\alpha_l\in\R_{>0}$ and $\kappa_l\in\N$ and $\vartheta_l\in C^0(M,\R_{>0})$ by
\begin{equation} \label{kappadelta} \begin{split}
\alpha_l &\define \min\bigSet{a_i}{i\in\set{0,\dots,l-1}} \,,\\
\kappa_l &\define \max\bigSet{k_i}{i\in\set{0,\dots,l-1}} \,,\\
\vartheta_l(x) &\define \max\bigSet{\theta_i(x)}{i\in\set{0,\dots,l-1}} \,.
\end{split} \end{equation}
This yields for all $l\in\N$ and $u\in C^\infty(M,\R_{\geq0})$:
\[
\tilde{\Phi}_l(u) \leq \sup\Set{ \e^{-\alpha_lu(y)}\;\vartheta_l(y)\cdot \left(1+\sum_{j=0}^{\kappa_l}\bigabs{\nabla_\eta^ju}_\eta(y)\right) }{y\in K_{l+1}\without K_{l-2}} \,.
\]

For $l\in\N$, we define
\[
\tilde{\lambda}_l \define \frac{\tilde{\eps}_l}{\sup\Set{\vartheta_l(y)}{y\in K_{l+1}\without K_{l-2}}} \in\R_{>0} \,.
\]
We choose a monotonically decreasing sequence $(\lambda_i)_{i\in\N}$ in $\R_{>0}$ with $\forall i\in\N\colon \lambda_i<\tilde{\lambda}_i$.

\smallskip
Due to \eqref{aim}, for any $u\in C^\infty(M,\R)$ which satisfies $u\geq\hat{w}$, the statement \eqref{pudelskern} is true if
\begin{equation} \label{aim2}
\forall i\in\N\colon \sup\Set{ \e^{-\alpha_iu(y)} \left(1+\sum_{j=0}^{\kappa_i}\bigabs{\nabla_\eta^ju}_\eta(y)\right) }{y\in K_{i+1}\without K_{i-2}} \leq \lambda_i \,.
\end{equation}

For each $i\in\N$, Lemma \ref{alpinist} yields an alpinist $\Theta_i=(\varphi_i[n])_{n\in\N}$ such that the set $G_{k_{i+1},\alpha_{i+1}}[\Theta_i]$ is bounded from above by some $C_i\in\R_{>0}$.

\smallskip
For $i\in\N$, we consider $L_i\define L_{i,\kappa_{i+1}}$. We define a sequence $b=(b_i)_{i\in\N}$ in $\N$ recursively by
\begin{equation} \label{bdef} \begin{split}
b_i \define \min\Big\{\; \beta\in\N \mathrel{}\Big|\mathrel{}
&\max\set{1+\beta,\; L_i\beta+C_iL_i} \leq \lambda_{i+1}\,\e^{\alpha_{i+1}\beta} \quad\AND\\[-0.2ex]
&\;\forall x\in K_i\without K_{i-1}\colon \hat{w}(x)\leq\beta \quad\AND\quad
\forall j\in\set{0,\dots,i-1}\colon b_j\leq\beta \;\Big\} \,;
\end{split} \end{equation}
this is well-defined because the set on the right-hand side is nonempty. By construction, $b$ is monotonically increasing. Hence the numbers $c_i\define b_{i+1}-b_i$ lie in $\N$.

\smallskip
We define a function $u\in C^\infty(M,\R)$ by
\[
u(x) \define \begin{cases}
b_i &\text{if $\exists i\in\N\colon x\in K_i\without(A_i\cup K_{i-1})$}\\
b_i +\varphi_i[c_i](\rho_i(x)) &\text{if $\exists i\in\N\colon x\in A_i$}
\end{cases} \;\;.
\]
Obviously $u$ is indeed a well-defined function. It is smooth because all $j$th derivatives with $j\geq1$ of $\varphi_i[c_i]$ vanish at $0$ and $1$ and because $b_i +\varphi_i[c_i](1) = b_i+c_i = b_{i+1}$.

Moreover, we have $u\geq\hat{w}>w$, because $\forall i\in\N\colon \forall x\in K_i\without K_{i-1}\colon u(x) \geq b_i \geq \hat{w}(x) > w(x)$.

\smallskip
Let $i\in\N$, let $y\in K_{i+1}\without K_{i-2}$. Then for some $\mu\in\set{-1,0,1}$, we have either $y\in K_{i+\mu}\without (A_{i+\mu}\cup K_{i+\mu-1})$ or $y\in A_{i+\mu}$. If $y\in K_{i+\mu}\without (A_{i+\mu}\cup K_{i+\mu-1})$, then the definition \eqref{bdef} of $b_{i+\mu}$ implies
\begin{equation} \label{aimpart1}
1+\sum_{j=0}^{\kappa_i}\bigabs{\nabla_\eta^ju}_\eta(y)
= 1 +b_{i+\mu}
\leq \lambda_{i+\mu+1}\,\e^{\alpha_{i+\mu+1}b_{i+\mu}}
\leq \lambda_i\,\e^{\alpha_iu(y)} \,,
\end{equation}
because $(\lambda_l)_{l\in\N}$ and $(\alpha_l)_{l\in\N}$ decrease monotonically and $u(y)=b_{i+\mu}\geq0$.

If $y\in A_{i+\mu}$, we consider $t\define\rho_{i+\mu}(y)$. Since $G_{\kappa_{i+\mu+1},\alpha_{i+\mu+1}}[\Theta_{i+\mu}]$ is bounded from above by $C_{i+\mu}$, Definition \ref{climbers} yields
\begin{equation} \label{Gstuff} \begin{split}
\e^{-\alpha_{i+\mu+1}\varphi_{i+\mu}[c_{i+\mu}](t)} \left(1+\sum_{j=0}^{\kappa_{i+\mu+1}}\abs{\varphi_{i+\mu}[c_{i+\mu}]^{(j)}(t)}\right)
&\leq \sup\left(G_{\kappa_{i+\mu+1},\alpha_{i+\mu+1}}[\Theta_{i+\mu}]\right)
\leq C_{i+\mu} \,.
\end{split} \end{equation}
By \eqref{kappadelta}, the sequence $(\kappa_l)_{l\in\N}$  increases monotonically, whereas $(\lambda_l)_{l\in\N}$ and $(\alpha_l)_{l\in\N}$ decrease monotonically. Since $\e^{\alpha_{i+\mu+1}\varphi_{i+\mu}[c_{i+\mu}](t)}\geq1$, we deduce from \eqref{bdef} and \eqref{Gstuff}:
\[ \begin{split}
\lambda_i\,\e^{\alpha_iu(y)}
&\;=\; \lambda_i\,\e^{\alpha_ib_{i+\mu}}\;\e^{\alpha_i\varphi_{i+\mu}[c_{i+\mu}](t)}\\
&\;\geq\; \lambda_{i+\mu+1}\,\e^{\alpha_{i+\mu+1}b_{i+\mu}}\;\e^{\alpha_{i+\mu+1}\varphi_{i+\mu}[c_{i+\mu}](t)}\\
&\;\geq\; L_{i+\mu}b_{i+\mu}\e^{\alpha_{i+\mu+1}\varphi_{i+\mu}[c_{i+\mu}](t)} +C_{i+\mu}L_{i+\mu}\;\e^{\alpha_{i+\mu+1}\varphi_{i+\mu}[c_{i+\mu}](t)}\\
&\;\geq\; L_{i+\mu}b_{i+\mu} +\frac{C_{i+\mu}L_{i+\mu}}{C_{i+\mu}} \left(1 +\sum_{j=0}^{\kappa_{i+\mu+1}}\abs{\varphi_{i+\mu}[c_{i+\mu}]^{(j)}(t)}\right)\\
&\;=\; L_{i+\mu}\cdot \left(1+b_{i+\mu}+\sum_{j=0}^{\kappa_{i+\mu+1}}\abs{\varphi_{i+\mu}[c_{i+\mu}]^{(j)}(t)}\right) \,.
\end{split} \]
Applying \eqref{chainrule} to the function $f\colon s\mapsto b_{i+\mu} +\varphi_{i+\mu}[c_{i+\mu}](s)$, we thus obtain
\begin{equation} \label{aimpart2} \begin{split}
1+\sum_{j=0}^{\kappa_i} \bigabs{\nabla^j_{\eta}u}_\eta(y)
&\;\leq\; 1+\sum_{j=0}^{\kappa_{i+\mu+1}} \bigabs{\nabla^j_{\eta}u}_\eta(y)
\;\leq\; L_{i+\mu}\cdot \left(1 +b_{i+\mu} +\sum_{j=0}^{\kappa_{i+\mu+1}}\abs{\varphi_{i+\mu}[c_{i+\mu}]^{(j)}(t)}\right)
\;\leq\; \lambda_i\,\e^{\alpha_iu(y)} \,.
\end{split} \end{equation}
The inequalities \eqref{aimpart1} and \eqref{aimpart2} imply that \eqref{aim2} is true for the function $u$ with $u\geq\hat{w}>w$ we have constructed. This shows already that there exists a function $u\in C^\infty(M,\R)$ with $u>w$ such that for every $i\in\N$, the inequality $\Phi_i(u)<\eps_i$ holds on $M\without K_i$.

Since $u$ satisfies by construction even
\[
\forall i\in\N\colon \sup\Set{ \e^{-\alpha_iu(y)} \left(1+\sum_{j=0}^{\kappa_i}\bigabs{\nabla_\eta^ju}_\eta(y)\right) }{y\in K_{i+1}\without K_{i-2}} < \tilde{\lambda}_i
\]
and $u>w$, there exists obviously a neighborhood $\mathcal{U}$ of $u$ in the fine $C^\infty$-topology on $C^\infty(M,\R)$ such that every $v\in\mathcal{U}$ satisfies $v>w$ and
\[
\forall i\in\N\colon \sup\Set{ \e^{-\alpha_iv(y)} \left(1+\sum_{j=0}^{\kappa_i}\bigabs{\nabla_\eta^jv}_\eta(y)\right) }{y\in K_{i+1}\without K_{i-2}} < \tilde{\lambda}_i \,.
\]
In particular, every $v\in\mathcal{U}$ satisfies $v>w$ and, for every $i\in\N$, $\Phi_i(v)<\eps_i$ on $M\without K_i$. Since real-analytic functions are fine-$C^\infty$-dense in $C^\infty(M,\R)$ (cf.\ e.g.\ \cite[Theorem A]{Shiga}), Theorem \ref{FlatzoomThemAll} is proved.
\end{proof}

Now we can prove our main result stated in the Introduction:
\begin{proof}[Proof of Theorem \ref{maingeneral}]
For $i\in\N$, consider the maps $\Psi_i,\Psi_i^\Fol,\Upsilon_i\colon C^\infty(M,\R) \to C^0(M,\R_{\geq0})$ defined by
\begin{align*}
\Psi_i(u) &\define \abs{\nabla^i_{g_0[u]}\Riem_{g_0[u]}}_{h_0[u]} \,,\\
\Psi_i^\Fol(u) &\define \abs{\nabla^i_{(g_0)_\Fol[u]}\Riem_{(g_0)_\Fol[u]}}_{(h_0)_\Fol[u]} \,,
&\Upsilon_i(u) &\define \abs{\nabla^i_{g_0[u]}\SecondFF^\Fol_{g_0[u]}}_{h_0[u]} \,.
\end{align*}
The Examples \ref{ex2}, \ref{ex3}, \ref{ex4} show that $\Psi_i$, $\Psi_i^\Fol$, $\Upsilon_i$ are flatzoomers. By \ref{quasiex1}, they are quasi-flatzoomers for $\Kex$. For $i\in\N$, we define $\Phi_i\colon C^\infty(M,\R) \to C^0(M,\R_{\geq0})$ by $\Phi_i(u)\define \Psi_i(u)+\Psi_i^\Fol(u)+\Upsilon_i(u)$. Example \ref{quasiex2} (see also \ref{ex5}) tells us that $\Phi_i$ is a quasi-flatzoomer.

Theorem \ref{FlatzoomThemAll}, applied to the sequence $(\Phi_i)_{i\in\N}$, yields a real-analytic function $u\colon M\to\R$ with $u>u_0$ such that for every $i\in\N$, the inequality $\Phi_i(u) < \eps_i$ holds on $M\without K_i$. Thus the statements \ref{gen1}, \ref{gen3}, \ref{gen5} of Theorem \ref{maingeneral} are true. If $(g_0)_\Fol$ (and thus also $g_0$) is not Riemannian, the proof of \ref{maingeneral} is now complete.

\smallskip
Otherwise we define a smooth compact exhaustion $\Kex'=(K_i')_{i\in\N}$ by $K_0'\define\leer$ and $\forall i\geq1\colon K_i'\define K_{i-1}$, define $(\eps_i')_{i\in\N}$ by $\eps_0'\define \frac{1}{\iota+1}$ and $\forall i\geq1\colon \eps_i'\define \eps_{i-1}$, and define $\forall i\geq1\colon \Phi_i'\define\Phi_{i-1}$. If $(g_0)_\Fol$, but not $g_0$, is Riemannian, then we consider $\Phi_0'\colon u\mapsto 1/\conv_{(g_0)_\Fol[u]}^\Fol$, which is a quasi-flatzoomer due to Theorem \ref{RiemannInj}. If $g_0$ is Riemannian, we consider $\Phi_0'\colon u\mapsto 1/\conv_{(g_0)_\Fol[u]}^\Fol +1/\conv_{g[u]}$, which is a quasi-flatzoomer due to Theorem \ref{RiemannInj} (applied also to the foliation whose only leaf is $M$) and Example \ref{quasiex2}.

Now Theorem \ref{FlatzoomThemAll}, applied to $\Kex'$ and $(\Phi_i')_{i\in\N}$ and $(\eps_i')_{i\in\N}$, shows that all statements of Theorem \ref{maingeneral} are true, because the convexity radii are by construction $\geq\iota+1\geq1$, which implies in particular completeness of the metrics. By \cite[Proposition IX.6.1]{Chavel}, this yields also the inequalities $\inj\geq2\conv$ (note that Chavel's $\conv$ is a priori $\geq$ the one we have defined at the beginning of \S\ref{radii}).
\end{proof}

The other results stated in Section \ref{intro} follow from Theorem \ref{maingeneral}, as explained there.

\bigskip
We end this article by stating explicitly, for future use elsewhere, one result about ordinary differential inequalities that has essentially been derived during the proof of Theorem \ref{FlatzoomThemAll}.

\begin{theorem} \label{OD}
Let $(\eps_i)_{i\in\N}$ and $(\alpha_i)_{i\in\N}$ be sequences in $\R_{>0}$, let $(m_i)_{i\in\N}$ be a sequence in $\N$, let $(P_i)_{i\in\N}$ be a sequence such that each $P_i$ is a real polynomial (whose degree may depend on $i$) in $m_i+1$ real variables. Let $w\in C^0(\coi{0}{\infty},\R)$. Then there exists a number $\mu\in\R$ such that for every $u_0\in\coi{\mu}{\infty}$, there is a function $u\in C^\infty(\coi{0}{\infty},\R)$ with the following properties:
\begin{enumerate}[label=(\roman*)]
\item\label{OD1} $u(0) = u_0$.
\item\label{OD2} For each $i\in\N$, \;$u$ is constant on the interval $\cci{i}{i+\tfrac{1}{2}}$.
\item\label{OD3} $u>w$.
\item\label{OD4} $\forall i\in\N\colon \forall x\in\cci{i}{i+1}\colon P_i\left(u(x),u'(x),\dots,u^{(m_i)}(x)\right) < \eps_i\e^{\alpha_iu(x)}$.
\end{enumerate}
\end{theorem}

\noindent
\emph{Remark 1.} In particular, the ordinary differential inequality \ref{OD4} can be solved globally with initial values $u(0)$ and $\forall i\geq1\colon u^{(i)}(0)=0$ whenever $u(0)$ is sufficiently large. In contrast, the results of \cite{MN} show that even in simple special cases, the inequality \ref{OD4} can\emph{not} be solved with $\forall i\geq1\colon u^{(i)}(0)=0$ for arbitrary initial values $u(0)$ that satisfy $P_0(u(0),0,\dots,0) < \eps_0\e^{\alpha_0u(0)}$ (the properties \ref{OD2}, \ref{OD3} do not matter for this conclusion).

\medskip\noindent
\emph{Remark 2.} The polynomials $P_i$ are assumed to have constant coefficients here, for simplicity. But since they may depend on the interval $\cci{i}{i+1}$, an inequality of the form
\[
\forall x\in\coi{0}{\infty}\colon P(x)\left(u(x),u'(x),\dots,u^{(m)}(x)\right) < \eps(x)\e^{\alpha(x)u(x)} \,,
\]
for a polynomial-valued function $P\in C^0\big(\coi{0}{\infty},\Poly{m+1}{d}\big)$ and functions $\eps,\alpha\in C^0(\coi{0}{\infty},\R_{>0})$, can always be strengthened to an inequality of the form \ref{OD4} and can then be solved using the theorem.

\begin{proof}[Sketch of proof of Theorem \ref{OD}]
For $M\define\coi{0}{\infty}$, we consider the smooth compact exhaustion $(K_i)_{i\in\N}$ with $K_i\define\cci{0}{i+1}$; this $M$ is a manifold-with-boundary, but the boundary does not cause any problem. After replacing $P_i$ by $P_i^2+1$ if necessary, we may assume that all $P_i$ are $\geq0$. The maps $\Phi_i\colon C^\infty(M,\R)\to C^\infty(M,\R_{\geq0})$ given by $\Phi_i(u)(x)\define \e^{-\alpha_i u(x)} P_i\left(u(x),u'(x),\dots,u^{(m_i)}(x)\right)$ are obviously (quasi-)flatzoomers. Revisiting the proof of Theorem \ref{FlatzoomThemAll} for our given data $(\Phi_i)_{i\in\N}$, $(\eps_i)_{i\in\N}$, $w$, we choose the interior collar neighborhoods $A_i = \cci{i+\frac{1}{2}}{i+1}$. Clearly there exists a number $\mu\in\R$ such that for every $u_0\in\coi{\mu}{\infty}$, we can choose the sequence $b$ with $b_0=u_0$. The constructed function $u\in C^\infty(M,\R)$ satisfies \ref{OD1}--\ref{OD4}.
\end{proof}


\end{document}